\documentclass{amsart}
\usepackage[utf8]{inputenc}

\usepackage{enumerate}
\usepackage{enumitem}
\usepackage{multicol}
\usepackage{amssymb,amsmath}
\usepackage{amsthm}
\usepackage{turnstile}
\usepackage{textcomp}
\usepackage{graphicx}
\usepackage{subcaption}
\usepackage{mathtools}
\usepackage{wrapfig}
\usepackage{hhline}
\usepackage{array,multirow}
\usepackage{xcolor}
\usepackage{float}
\usepackage{geometry}
\usepackage{tikz}

\usepackage{hyperref}
\usepackage[normalem]{ulem}

\usetikzlibrary{positioning,calc}

\theoremstyle{plain}
\newtheorem{theorem}{Theorem}[section]

\newtheorem{definition}[theorem]{Definition}

\newtheorem{fact}[theorem]{Fact}

\DeclareMathOperator{\prcat}{\mathtt{PRCatSpec}}
\DeclareMathOperator{\cone}{\mathtt{Cone}}

\title{Primitive recursive categoricity spectra of functional structures}

\author{Nikolay Bazhenov}
\address{Novosibirsk State University, Novosibirsk, Russia}
\address{Nazarbayev University, Astana, Kazakhstan}
\email{nickbazh@yandex.ru}

\author{Heer Tern Koh}
\address{Nanyang Technological University, Singapore}
\email{heertern001@e.ntu.edu.sg}

\author{Keng Meng Ng}
\address{Nanyang Technological University, Singapore}
\email{kmng@ntu.edu.sg}

\thanks{Ng was supported by the Ministry of Education, Singapore, under its Academic Research Fund Tier 2 (MOE-T2EP20222-0018) and Academic Research Fund Tier 1 (RG104/24). We also thank David Belanger for the many helpful discussions.}

\date{\today}

\begin{document}

\begin{abstract}
For the notion of degree of categoricity, we study an analogous notion for punctual structures. We show that such notions coincide for non-$\Delta_{1}^{0}$-categorical injection structures, and construct an example of a $\Delta_{1}^{0}$-categorical injection structure for which these notions differ. Additionally, we also show that in every non-zero c.e.~Turing degree, there exists a PR-degree that is low for punctual isomorphism (to be defined), and also a PR-degree that is a degree of punctual categoricity.
\end{abstract}

\maketitle

\section{Introduction}

One of the key research programs in computable structure theory is the investigation of the algorithmic content of isomorphisms between computable structures (first defined in \cite{rab60,mal61}). More specifically, given some computable structure $\mathcal{A}$, researchers attempt to find a Turing degree $\mathbf{d}$ so that for any two computable presentations $B,C\cong\mathcal{A}$, $\mathbf{d}$ computes an isomorphism between $B$ and $C$. For such a degree $\mathbf{d}$ and structure $\mathcal{A}$, we say that $\mathcal{A}$ is $\mathbf{d}$-\emph{computably categorical}. Intuitively, this is a measure of the amount of non-computable information required in order to algorithmically produce isomorphisms between copies of $\mathcal{A}$. Under such a setting, the algorithmic content of isomorphisms on various classes of structures have been studied (see for example \cite{gd80,remmel81b,remmel81,mccoy03,cchm06,cchm09,chr14,fro15,baz17}).

Conversely, given a Turing degree $\mathbf{d}$, one can also ask if such a degree is able to compute isomorphisms between certain structures. To make such a question meaningful, one should really consider only the structures $\mathcal{A}$ for which $\mathbf{d}$ is the least degree that is capable of producing isomorphisms of $\mathcal{A}$. Otherwise, the answer is always trivially yes, as there exist structures that are \emph{computably categorical}; structures which have computable isomorphisms between any two of their presentations. In the event that a structure $\mathcal{A}$ exists for which $\mathbf{d}$ is the least degree that computes an isomorphism between any two given presentations of $\mathcal{A}$, then we say that $\mathbf{d}$ is a \emph{degree of categoricity} \cite{fkm10,gon11}. We refer to, e.g., \cite{fkm10,CFS-13,BKY-18,CS-19,CDH-20,Baz-21,CN-22,CR-24} for known results on degrees of categoricity. On the other hand, if $\mathbf{d}$ is such that any structure that is $\mathbf{d}$-computably categorical is also computably categorical, then we say that $\mathbf{d}$ is \emph{low for computable isomorphism} \cite{fs14}. Intuitively, this means that $\mathbf{d}$ is `computationally weak' in the sense that it provides no additional information when being used to compute isomorphisms. 

Closely related to computable structure theory is the study of \emph{punctual structures}. These are structures with domain $\omega$ and whose operations and relations are uniformly primitive recursive \cite{kmn17}. The study of such structures was motivated by the phenomenon that `feasible' presentations may be extracted from computable ones simply by eliminating the use of unbounded search \cite{grigorieff90,cr91,cr92,cdru09}. Thus, suggesting that isolating the minimisation operator provides insight into the gap between `feasible' presentations and computable ones. Such an approach has attracted a fair amount of attention and has proven to provide a different perspective from computable structure theory (see \cite{bdkmn19,dmn21} for surveys regarding punctual structures). Following this pattern, we study the analogue of degrees of categoricity in the present article.

\subsection{Preliminaries}\label{sec:prelim}

Recall that the \emph{degree of categoricity} of a computable structure is the least Turing degree $\mathbf{d}$ which computes an isomorphism between any two computable presentations of the structure. Evidently, to show that some structure has degree of categoricity $\mathbf{d}$, one must first show that $\mathbf{d}$ is sufficient to produce isomorphisms between any two computable presentations of the given structure. Second, to show that such a degree is `sharp', one typically constructs two computable presentations so that any isomorphism $f$ between these presentations is such that $f\geq_{T}\mathbf{d}$. Roughly speaking, the degree of categoricity of a computable structure measures how algorithmically complex isomorphisms between computable presentations of the given structure are. In order to study the primitive recursive analogue of such notions, we require the following definition.

\begin{definition}[\cite{km21}]\label{def:pr}
    Let $f,g:\omega\to\omega$ be total functions. $f$ reduces to $g$, written $f\leq_{PR}g$ if there is some primitive recursive scheme $\Psi$ such that $\Psi^{g}=f$. If we further have that $g\leq_{PR}f$, then we write $f\equiv_{PR}g$. The induced degree structure shall be referred to as the PR-degrees.
\end{definition}

Just like how Turing degrees are a measure of non-computability, the PR-degrees analogously measures the primitive recursive content of various functions. Intuitively, if $f\leq_{PR}g$, then $g$ grows faster than $f$ does, and is thus more `primitive recursively complex' than $f$. Under this setting, and following the ideas and terminology from computable structure theory, we define:

\begin{definition}[\cite{bk21}]\label{def:catspec}
    For a given punctual structure $A$, we use $\prcat(A)$ to denote the set containing all PR-degrees $\mathbf{d}$ such that for any punctual structure $B\cong A$, there is an isomorphism $f:A\to B$ and $\mathbf{d}\geq_{PR}f,f^{-1}$.
\end{definition}

Since the inverse of a primitive recursive function is not necessarily primitive recursive, it follows that for a given function $f$, $f$ and $f^{-1}$ may belong to different PR-degrees. As such, when defining the primitive recursive categoricity spectrum of a structure as above, we require that each PR-degree $\mathbf{d}\in\prcat(A)$ primitively recursively computes both an isomorphism and its inverse.

In computability theory, the Turing degrees $\mathbf{0},\mathbf{0}',\mathbf{0}'',\dots$ serves as a `spine' to classify and rank the algorithmic content of various objects or processes. In a similar vein, we define the following collections of PR-degrees:

\begin{definition}[\cite{bk21}]\label{def:cone}
    For each computable ordinal $\alpha>0$, $\cone(\Delta_{\alpha}^{0})$ is the set containing all PR-degrees $\mathbf{d}$ such that $\mathbf{d}\geq_{PR}f$ for any total $\Delta_{\alpha}^{0}$-function $f$.
\end{definition}

From Definitions \ref{def:cone} and \ref{def:catspec}, it is not difficult to obtain the following:

\begin{fact}\label{fact:conesubspec}
    If $A$ is $\Delta_{\alpha}^{0}$-categorical, then $\cone(\Delta_{\alpha}^{0})\subseteq\prcat(A)$.
\end{fact}

In some sense, this is the analogue of $\Delta_{\alpha}^{0}$-categoricity of computable structures for punctual structures. In order to show that this bound is sharp, i.e., that $\Delta_{\alpha}^{0}$ is the least primitive recursive degree of categoricity for $A$, we show that $\prcat(A)\subseteq\cone(\Delta_{\alpha}^{0})$, and hence $\prcat(A)=\cone(\Delta_{\alpha}^{0})$. This means that the PR-degrees which primitively recursively compute isomorphisms of $A$ exactly coincide with the PR-degrees which primitively recursively compute $\Delta_{\alpha}^{0}$-functions, and thus, is analogous to the notion of degree of categoricity for computable structures.

Finally, we say that a PR-degree $\mathbf{d}$ is \emph{low for punctual isomorphism} if for any structure $A\cong B$ such that $f:A\to B$ is an isomorphism and $f\oplus f^{-1}\leq_{PR}\mathbf{d}$, then $A$ and $B$ are punctually isomorphic. We say that a PR-degree $\mathbf{d}$ is a \emph{degree of punctual categoricity if there exists a punctual structure $A$ such that $\mathbf{d}$ is the least degree in the spectrum $\prcat(A)$. To our best knowledge, the only known example of degrees of punctual categoricity is provided by the following result: if the graph of a total function $f(x)$ is primitive recursive, then $\deg(f)$ is a degree of punctual categoricity (Proposition~2.7 in~\cite{KMN-17-AL}). We also cite~\cite{km21,KK-24} for related results.} 

\

In this article, we study the primitive recursive categoricity spectra of injection structures, showing that they generally coincide with the analogous notions of $\Delta_{n}^{0}$-categoricity. Nonetheless, in Section \ref{sec:pathological}, we give an example of a computably categorical injection structure $A$ for which $\prcat(A)\not\subseteq\cone(\Delta_{1}^{0})$, contrasting with the results obtained in Section \ref{sec:inj} and the companion paper \cite{bbkn25}. In the final section of this article, we also show that within each non-zero c.e.~Turing degree, there are PR-degrees which exhibit `contrasting' behaviour; in every non-zero c.e.~Turing degree, there is a PR-degree low for punctual isomorphism and also a PR-degree that is a degree of punctual categoricity.

\section{Injection Structures}\label{sec:inj}

An injection structure is a structure with a single unary function symbol, such that the function is injective. It is evident that such structures may be characterised by the number and types of orbits it contains. More specifically, each orbit is either finite, or is an $\omega$-chain, or is a $\zeta$-chain. We first show that such structures also possess punctual presentations.

\begin{theorem}[folklore]\label{thm:puncinj}
    Let $A$ be a computable injection structure with at least one infinite orbit. Then $A$ has a punctual presentation.
\end{theorem}

\begin{proof}
Let $A$ be a computable injection structure with at least one infinite orbit. Non-uniformly fix $N_{0},N_{1}\in\omega\cup\{\infty\}$ as the number of orbits in $A$ of type $\omega$ and $\zeta$ respectively. A punctual presentation of $P$ of $A$ may be obtained as follows. At each stage $s$, enumerate $n_{0}+n_{1}$ many elements into $P$, where $n_{i}=\min\{s,N_{i}\}$, extending the infinite orbits in the obvious way. Since at least one of $N_{0},N_{1}$ is nonzero, this ensures that $P$ remains punctual. For any orbits of finite size, we simply wait for their size to be revealed in $A$ (this is a c.e.~ wait), once revealed, we enumerate a finite cycle of the same size into $P$.
\end{proof}

Recall that a computable injection structure $A$ is relatively $\Delta^0_1$-categorical if and only if $A$ is $\Delta^0_1$-categorical if and only if $A$ has only finitely many infinite orbits \cite[Corollary 2.5]{chr14}.

\begin{theorem}\label{thm:d1inj}
    If $A$ is a relatively $\Delta_{1}^{0}$-categorical injection structure with at least one infinite orbit and infinitely many finite orbits, then $\prcat(A)=\cone(\Delta_{1}^{0})$.
\end{theorem}

\begin{proof}
From Fact \ref{fact:conesubspec}, we have that $\cone(\Delta_{1}^{0})\subseteq\prcat(A)$. In order to prove the remaining inclusion, for any total computable $g$, we produce punctual injection structures $A$ and $B$ such that any isomorphism $h:A\to B$ is such that $g\leq_{pr}h\oplus h^{-1}$.

\

\emph{Encoding $g(0)$:} Since $g$ is a total computable function, all we need is to encode is the stage $s$ at which $g(0)[s]\downarrow$. We shall do this by enumerating finite orbits into one structure $A$ `quickly' while `delaying' their enumeration into $B$ until after stage $s$ is reached. The main tension here is that we have to keep both $A$ and $B$ punctual; we cannot arbitrarily delay enumerating elements. Since we assumed that $A$ has at least one infinite orbit, then while waiting for $g(0)\downarrow$, we enumerate elements into the infinite orbit to keep $A$ and $B$ punctual. Eventually, $A$ has to reveal some finite orbit. Once this finite orbit is found and $g(0)$ has converged, then we enumerate the same finite orbit into the structure $B$. If $h:A\to B$ is an isomorphism, then by computing $h$ on the elements within the finite orbit, we may recover the stage where $g(0)$ converges.

\

\emph{Encoding $g(x)$:} The obvious generalisation of the strategy described above would be to compute $h$ on the first $x$ many finite orbits revealed in $A$, and then arguing that at least one of these images must have index larger than the stage at which $g(x)$ converges. However, we note here that this process is not necessarily primitive recursive in $h\oplus h^{-1}$ and $x$. In particular, an arbitrary punctual injection structure $A$ could enumerate finite orbits only at extremely sparse stages, meaning that the index of the $x^{th}$ finite orbit in $A$ cannot be obtained in a primitive recursive way.

To solve this issue, in addition to encoding the stage at which $g(x)$ converges, we also encode the stage $G(x+1)$ at which $A$ enumerates the $x+1$-th finite orbit. More specifically, we enumerate the $x$-th finite orbit into $B$ only after the $x+1$-th finite orbit has been enumerated into $A$.

Now, we may define $\Psi^{h}$ as follows. Non-uniformly fix $G(0)$, the stage at which $A$ reveals the first finite orbit. From this stage, we may recover primitive recursively the indices of the elements in this finite orbit and then compute $h$ on them. The image must produce indices large enough such that $A$ has enumerated the next finite orbit and $g(0)$ has converged. Recursively, we may suppose that $\Psi^{h}(x)=g(x)\oplus G(x+1)$. To define $\Psi^{h}(x+1)$, use $G(x+1)$ to recover the indices of the $x+1$-th finite orbit in $A$ and compute $h$ on all elements of the first $x+1$ many finite orbits in $A$. Since $B$ never enumerates the $x+1$-th finite orbit until $A$ enumerates the $x+2$-th finite orbit and $g(x+1)$ converges, one of the $h$-images must produce an index $\geq G(x+2),g(x+1)$. That is to say, $h\geq_{pr}g\oplus G\geq_{pr}g$.
\end{proof}

\

Observe that the existence of an infinite orbit in $A$ allows us to `waste time' arbitrarily and delay enumerating the finite orbits into $B$. The proof above can thus be easily modified to fit the cases where there is some mechanism to `waste time'. For instance, the result should also hold if $A$ is relatively $\Delta_{1}^{0}$-categorical and has some finite size repeated infinitely often. Recall that the remaining case for $\Delta_{1}^{0}$-categorical injection structures are those with no infinite orbits and only finitely many orbits of each finite size. There are certain special cases in which we can apply a modification of the proof above to obtain $\prcat(A)\subseteq\cone(\Delta_{1}^{0})$, but surprisingly, it does not hold in general (see Theorem \ref{thm:inj} for the details).

A computable injection structure $A$ is relatively $\Delta^0_2$-categorical if and only if $A$ is $\Delta^0_2$-categorical if and only if $A$ has finitely many $\zeta$ chains or finitely many $\omega$ chains \cite[Corollary 3.3]{chr14}.

\begin{theorem}\label{thm:d2inj}
    If $A$ is a relatively $\Delta_{2}^{0}$-categorical injection structure which is not relatively $\Delta^0_1$-categorical, then $\prcat(A)=\cone(\Delta_{2}^{0})$.
\end{theorem}

\begin{proof}
If an injection structure $A$ is relatively $\Delta^0_2$-categorical, but not relatively $\Delta^0_1$-categorical, then its infinite orbits must be one of the following types:
\begin{enumerate}[label=(\roman*)]
    \item Finitely many $\zeta$ chains and infinitely many $\omega$ chains.
    
    \item Only finitely many $\omega$ chains and infinitely many $\zeta$ chains.
\end{enumerate}

Applying Fact \ref{fact:conesubspec}, we clearly have that $\cone(\Delta_{2}^{0})\subseteq\prcat(A)$. It remains to prove that $\prcat(A)\subseteq\cone(\Delta_{2}^{0})$. That is, given $g$, a total $\Delta_{2}^{0}$ function, we define two punctual injection structures $A$ and $B$ such that $g\leq_{pr}h\oplus h^{-1}$ for any isomorphism $h:A\to B$.

Fix a primitive recursive approximation $g^{*}$ such that $\lim_{s}g^{*}(x,s)=g(x)$ for all $x$, and fix some injection structure of Type (i) that has a punctual presentation. In $A$ and $B$, we never enumerate any finite orbits until they are observed to close in the given punctual (or computable) injection structure. More specifically, we only enumerate $\omega$ chains while waiting for finite orbits to be enumerated into the given punctual injection structure. Similarly, we may non-uniformly fix the finitely many $\zeta$ chains and enumerate them in a standard way during the construction. This ensures that $A$ and $B$ are of the same isomorphism type as the given punctual injection structure. Obviously, we shall use the $\omega$ chains to encode the stabilising stages $s_{x}$; $g^{*}(x,s_{x})=g^{*}(x,s)$ for all $s\geq s_{x}$.

\

\emph{Encoding $s_{x}$:} As mentioned above, we construct punctual structures $A=(\omega,f_{A})$ and $B=(\omega,f_{B})$ such that any isomorphism from $A$ to $B$ encodes $g$. In $A$, all $\omega$ chains will be standard; this allows us to access the left-most element of each $\omega$ chain in a primitive recursive way. We denote these elements as $a_{0,0},a_{1,0},a_{2,0},\dots$. At each stage $s$, enumerate $a_{s,0}$ and for each $a_{i,j}$ currently in $A$ where $f_{A}(a_{i,j})$ has yet to be defined, enumerate $a_{i,j+1}\coloneqq f_{A}(a_{i,j})$. This ensures that $f_{A}$ is primitive recursive and $A$ is punctual.

In $B$, we shall use the left-most element of each $\omega$ chain in $B$ to encode the stabilising stage $s$ of $g^{*}(x,s)$ for various $x$. At stage $0$, enumerate the element $b_{0}$ and let $b^{l}_{0,0}=b^{r}_{0,0}=b_{0}$. The idea is that $b^{l}_{i,s}$ and $b^{r}_{i,s}$ will denote the current left-most and right-most elements of the $i^{th}$ $\omega$ chain at stage $s$ respectively. At each stage $s$, recursively suppose that $b^{l}_{i,s-1}$ and $b^{r}_{i,s-1}$ have been defined for $i<s$. If $g^{*}(x,s)\neq g^{*}(x,s-1)$, then do the following.
\begin{enumerate}
    \item For each $j$ where $x<j<s$, define $f_{B}(b^{r}_{j-1,s-1})=b^{l}_{j,s-1}$.

    \item For each $j<x$, let $b^{l}_{j,s}=b^{l}_{j,s-1}$, and $b^{r}_{j,s}=b^{r}_{j,s-1}$.

    \item Let $b^{l}_{x,s}=b^{l}_{x,s-1}$ and $b^{r}_{x,s}=b^{r}_{s-1,s-1}$.

    \item Pick fresh indices $i_{x+1},i_{x+2},\dots,i_{s}$, and enumerate the elements $b_{i_{x+1}},b_{i_{x+2}},\dots,b_{i_{s}}$, defining $b^{l}_{j,s}=b^{r}_{j,s}=b_{i_{j}}$ for each $x<j\leq s$.
\end{enumerate}
Intuitively, if $g^{*}(x,s)\neq g^{*}(x,s-1)$, then we `glue' the $j^{th}$ chain for $x<j<s$ to the $x^{th}$ chain at stage $s$, while leaving the first $x$ many chains the same. This ensures that any element in the subsequent chains will have indices $>s$. Finally, to keep $f_{B}$ primitive recursive, for each element $b\in B$ enumerated at stage $s-1$, enumerate some new element $b'$ and let $f_{B}(b)=b'$ if $f_{B}(b)$ is not yet defined. In addition, if $b$ is such that $b=b^{r}_{i,s}$ for some $i$, re-define $b^{r}_{i,s}$ to be $b'$.

\

\emph{Verification:} Recall that any finite cycle or any $\zeta$ chain will be enumerated by `copying' the given punctual injection structure. Thus, to verify that the construction produces the correct isomorphism type, it suffices to show that the procedure above produces infinitely many $\omega$ chains and nothing else. First, observe that for any $b$ enumerated by the procedure above, $f_{B}(b)$ exists; if $b$ was enumerated at stage $s$, by stage $s+1$, $f_{B}(b)$ must have been defined. Thus, there is no right-most element produced by the procedure above. Next, to show that each chain has a left-most element, we argue that for each $i\in\omega$, there exists $s\in\omega$ such that for all $t\geq s,\,b^{l}_{i,s}=b^{l}_{i,t}$. It follows directly from the description that $b^{l}_{0,t}=b^{l}_{0,0}$ for all $t\geq 0$. For $i>0$, consider the stage $s^{*}$ such that $g^{*}(x,s^{*})=g^{*}(x,t)$ for all $t\geq s^{*}$ and for all $x<i$. Since $g$ is a total $\Delta_{2}^{0}$ function, such a stage $s^{*}$ must exist. Following the construction, $f_{B}^{-1}(b^{l}_{i,s})$ is only ever defined when $g^{*}(x,s)\neq g^{*}(x,s-1)$ for some $x<i$. In particular, for all stages $t\geq s^{*}+1$, we have that $g^{*}(x,t)=g^{*}(x,t-1)$ and so, $b^{l}_{i,t}=b^{l}_{i,s^{*}+1}$.

Finally, let $h:A\to B$ be an isomorphism. Define $\Psi^{h}$, a primitive recursive scheme as follows. Given $x$, compute $g^{*}(x,\max\{h(a_{i,0})\mid i\leq x+1\})$. By pidgeonhole principle, at least one of $h(a_{i,0})=b^{l}_{j,s}$ for some $j\geq x+1$ and for some $s\in\omega$. It then follows that $\max\{h(a_{i,0})\mid i\leq x+1\}$ is at least as big as the index of the left-most element of the $x+1$-th $\omega$ chain in $B$, which encodes a stage large enough such that $g^{*}$ has stabilised for all $y\leq x$. In particular, $g^{*}(x,\max\{h(a_{i,0})\mid i\leq x+1\})=g(x)$.



For injection structures of type (ii), we adopt a similar strategy. The only difference is that we must produce infinitely many $\zeta$ chains and nothing else when encoding $g$. That is, in addition to the procedure defined above, at each stage $s$, for each $b$ currently defined as $b^{l}_{i,s}$ for some $i$, enumerate a new element $b'$, define $f_{B}(b')=b$, and re-define $b^{l}_{i,s}=b'$. In $A$, we take a similar approach and extend the chains to the left at each stage by adding $a_{i,j}$ for $j<0$ and letting $f_{A}(a_{i,j})=a_{i,j+1}$ as before.

Using the modified construction, it is not too difficult to see that instead of $\omega$ chains, we now produce infinitely many $\zeta$ chains and nothing else. Furthermore, for each $x$, the $x+1$-th $\zeta$ chain in $B$ has the property that all elements within it have indices $\geq s$, where $s$ is such that for all $t\geq s$ and for all $y\leq x$, $g^{*}(y,t)=g^{*}(y,s)$. For the same reasons as before, $g(x)=g^{*}(x,\max\{h(a_{i,0})\mid i\leq x+1\})$, and thus $g\leq_{pr}h$. Theorem~\ref{thm:d2inj} is proved.
\end{proof}

\

Note that every computable injection structure is relatively $\Delta^0_3$-categorical \cite[Theorem 3.6]{chr14}.

\begin{theorem}\label{thm:d3inj}
    If $A$ is a relatively $\Delta_{3}^{0}$-categorical injection structure which is not relatively $\Delta^0_2$-categorical, then $\prcat(A)=\cone(\Delta_{3}^{0})$.
\end{theorem}

\begin{proof}
Let $A$ be a relatively $\Delta_{3}^{0}$-categorical, not-$\Delta^0_2$-categorical injection structure with a punctual presentation. Just as before, we focus only on the $\omega$ chains and $\zeta$ chains, enumerating finite orbits into our structures $A=(\omega,f_{A})$ and $B=(\omega,f_{B})$ only when they are shown to close in the given punctual copy. Let $g$ be a $\Delta_{3}^{0}$ function and let $g^{*}(x,s,t)$ be a primitive recursive approximation to $g$; $\lim_{s}\lim_{t}g^{*}(x,s,t)=g(x)$. The strategy to encode $g$ will be split into two parts. The first is to encode for each pair $(x,s)$, the stage $t_{x,s}$ such that $g^{*}(x,s,t)=g^{*}(x,s,t_{x,s})$ for all $t\geq t_{x,s}$. The second is to encode the stage $s_{x}$ for which $\lim_{t}g^{*}(x,s_{x},t)=g(x)$. Then using the isomorphisms $h:A\to B$ and $h^{-1}:B\to A$, we recover both $s_{x}$ and $t_{x,s_{x}}$ respectively, and compute $g^{*}(x,s_{x},t_{x,s_{x}})$.

\

\emph{Encoding $t_{x,s}$:} In $A$, for each $(x,s)$, we keep standard $\omega$ chains with left-most elements $a_{\langle x,s\rangle,0}$. The idea is to encode $t_{x,s}$ using $h(a_{\langle x,s\rangle,0})$. Intuitively, this means that we should keep the indices of the left-most elements of each $\omega$ chain in $B$ `large enough'. In $B$, we begin with $\omega$ chains with left-most elements $b_{0,0},b_{1,0},b_{2,0},\dots$ respectively. At each stage $t$ of the construction, we track the left-most elements of each of these $\omega$ chains in $B$, denoted by $b^{l}_{0,t}$. For each $i$, let $b^{l}_{i,0}=b_{i,0}$. At stage $t$, if $g^{*}(x,s,t)\neq g^{*}(x,s,t-1)$, then for each $i\geq\langle x,s\rangle$, enumerate some element $b^{*}$ with some index $\geq t$ and define $f_{B}(b^{*})=b^{l}_{i,t-1}$ and let $b^{l}_{i,t}=b^{*}$. Since $\lim_{t}g^{*}(x,s,t)$ exists for each $(x,s)$, there must be some stage $t^{*}$ for which $\lim_{t}g^{*}(x',s',t)=g^{*}(x',s',t^{*})$ for each $x'\leq x$ and $s'\leq s$. After such a stage $t^{*}$, $b^{l}_{i,t}$ for each $i\leq\langle x,s\rangle$ never again changes. Furthermore, $b^{l}_{\langle x,s\rangle,t^{*}}$ must have index $\geq t^{*}$.

\

\emph{Encoding $s_{x}$:} Let $P(x,s)$ be the predicate ``$\lim_{t}g^{*}(x,s,t)=g(x)$''. Since this is equivalent to ''for all $s'\geq s,\,\lim_{t}g^{*}(x,s,t)=\lim_{t}g^{*}(x,s',t)$'', $P(x,s)$ is a $\Pi_{2}^{0}$ predicate. We may further assume that there is a unique $s$ for which $P(x,s)$ holds by requiring $s\geq x$ to be the least, and that for any $x<y$, if $P(x,s_{x})$ and $P(y,s_{y})$ are both true, then $s_{x}<s_{y}$. To keep the strategies disjoint, we encode the various $s_{x}$ using the infinitely many $\zeta$ chains. In $B$, we enumerate standard copies of $\zeta$ chains, indexed by $c_{i,j}$ for $i,j\in\omega$, defining $f_{B}(c_{i,1})=c_{i,0}$, $f_{B}(c_{i,2j})=c_{i,2(j+1)}$ and $f_{B}(c_{i,2j+3})=c_{i,2j+1}$. Within $A$, let the possible $\zeta$ chains be encoded by $d_{i,j}$ for $i,j\in\omega$. By a careful definition of $f_{A}$, for a fixed $s$, the collection of $d_{s,j}$ becomes a $\zeta$ chain iff $P(0,s)\vee P(1,s)\vee\dots\vee P(s,s)$ holds. Intuitively, we `grow' the chain of $d_{s,j}$ to the left only at stages where one of $P(x,s)$ currently holds for some $x\leq s$. Since $P(x,s)$ is $\Pi_{2}^{0}$, if it holds, then it must look to be true at infinitely many stages. Which means that at infinitely many stages, we extend the chain to the left, resulting in a $\zeta$ chain. Conversely, if there are infinitely many stages where we extend the chain containing $d_{s,0}$ to the left, at least one of $P(0,s),P(1,s),\dots,P(s,s)$ must look to hold at infinitely many stages. That is to say, one of $P(x,s)$ is true for $x\leq s$. To recover $s_{x}$, we compute $h^{-1}(c_{i,0})$ for sufficiently many $i$, and we know that such an image must map to some $d_{s,j}$ contained in a $\zeta$ chain. Using pigeonhole principle, we then conclude that one such $d_{s,j}$ must have $s\geq s_{x}$.

\

\emph{Construction:} We summarise the various types of elements in $A$ and $B$ below.
\begin{center}
\begin{tabular}{l|l}
    Element type & Purpose \\ \hhline{=|=}
    $a_{i,j}\in A$, for all $i,j\in\omega$. & Standard $\omega$ chains; encoding $t_{x,s}$ where $i=\langle x,s\rangle$.\\
    $b_{i,j}\in B$, for all $i\in\omega$ and all $j\geq i$. & Non-standard $\omega$ chains to be images for $h(a_{i,0})$.\\ \hline
    $c_{i,j}\in B$, for all $i,j\in\omega$. & Standard $\zeta$ chains; encoding $s_{x}$ where $i=\langle x,s\rangle$.\\
    $d_{i,j}\in A$, for all $i,j\in\omega$. & $\zeta$ chain iff there exists $x\leq i$ such that $P(x,i)$ holds.
\end{tabular}
\end{center}
At stage $0$, enumerate $a_{0,0},d_{0,0}$ into $A$, $b_{0,0},c_{0,0}$ into $B$. Define $b^{l}_{0,0}=b^{r}_{0,0}=b_{0,0}$ and $d^{l}_{0,0}=d^{r}_{0,0}=d_{0,0}$. At stage $s$, suppose that we have already defined for each $i<s$, $b^{l}_{i,s-1},b^{r}_{i,s-1},d^{l}_{i,s-1}$, and $d^{r}_{i,s-1}$.
\begin{description}
    \item[Step 1] Enumerate $a_{s,j}$ for each $j<s$ and $a_{i,s}$ for each $i\leq s$ into $A$. Then define $f_{A}(a_{s,j})=a_{s,j+1}$ for each $j<s$ and define $f_{A}(a_{i,s-1})=a_{i,s}$ for each $i\leq s$. This ensures that at each stage $s$, $a_{i,j}$ for all $i,j\leq s$ has been enumerated and that $f_{A}$ is primitive recursive on all such elements.

    \item[Step 2] Enumerate $b_{s,s}$ into $B$ and define $b^{l}_{s,s}=b^{r}_{s,s}=b_{s,s}$. For each $i<s$, let $j_{i}$ be the least index such that $j_{i}>j$ for any $b_{i,j}$ currently enumerated into $B$. If there is some $\langle x,n\rangle=i<s$ and $g^{*}(x,n,s)\neq g^{*}(x,n,s-1)$, then enumerate $b_{k,j_{k}}$ and $b_{k,j_{k}+1}$ into $B$ for all $k$ such that $i\leq k<s$, defining
    \begin{itemize}
        \item $b^{r}_{k,s}=f_{B}(b^{r}_{k,s-1})=b_{k,j_{k}+1}$,
        \item $b^{l}_{k,s}=b_{k,j_{k}}$, and
        \item $f_{B}(b_{k,j_{k}})=b^{l}_{k,s-1}$.
    \end{itemize}
    Otherwise, only enumerate $b_{i,j_{i}}$ into $B$, define $b^{r}_{i,s}=f_{B}(b^{r}_{i,s-1})=b_{i,j_{i}}$ and $b^{l}_{i,s}=b^{l}_{i,s-1}$.

    \item[Step 3] Enumerate $c_{s,2j},c_{s,2j+1}$ for each $j<s$ and $c_{i,2s},c_{i,2s+1}$ for each $i\leq s$ into $B$. Now define $f_{B}(c_{i,2(s-1)})=c_{i,2s}$, and $f_{B}(c_{i,2s+1})=c_{i,2s-1}$ for each $i<s$, and
    $$
    c_{s,2s+1}\xrightarrow{f_{B}}c_{s,2s-1}\xrightarrow{f_{B}}\dots\xrightarrow{f_{B}}c_{s,1}\xrightarrow{f_{B}}c_{s,0}\xrightarrow{f_{B}}c_{s,2}\xrightarrow{f_{B}}\dots\xrightarrow{f_{B}}c_{s,2s}.
    $$
    At each stage $s$, we have that $B$ contains $c_{i,j}$ for each $i\leq s$ and each $j<2(s+1)$, and $f_{B}$ is primitive recursive on all such elements.

    \item[Step 4] Enumerate $d_{s,0}$ into $A$, and define $d^{l}_{s,s}=d^{r}_{s,s}=d_{s,0}$. For each $i<s$, let $j_{i}$ be the least index such that $j_{i}>j$ for any $d_{i,j}$ currently enumerated into $A$. If there is some $x\leq i<s$ such that $P(x,i)$ \emph{fires} (looks to be true at the current stage), then enumerate $d_{i,j_{i}}$ and $d_{i,j_{i}+1}$ into $A$, and define the following.
    \begin{itemize}
        \item $d^{r}_{i,s}=f_{A}(d^{r}_{i,s-1})=d_{i,j_{i}+1}$,
        \item $d^{l}_{i,s}=d_{i,j_{i}}$, and
        \item $f_{A}(d_{i,j_{i}})=d^{l}_{i,s-1}$.
    \end{itemize}
    Otherwise, only enumerate $d_{i,j_{i}}$ into $A$, and define $d^{r}_{i,s}=f_{A}(d^{r}_{i,s-1})=d_{i,j_{i}}$. Also let $d^{l}_{i,s}=d^{l}_{i,s-1}$.
\end{description}
Whenever a finite orbit in the given punctual injection structure is revealed, we also enumerate a finite orbit of the same size into the structures $A$ and $B$.

\

\emph{Verification:} We make the following observations.
\begin{enumerate}[label=(\roman*)]
    \item For any $x,s$, the collection $\{b_{\langle x,s\rangle,j}\mid j\geq\langle x,s\rangle\}$ forms an $\omega$ chain with left-most element $b_{\langle x,s\rangle,t}$, where for all $t'\geq t$ and for any $\langle x',s'\rangle\leq\langle x,s\rangle$, $g^{*}(x',s',t')=g^{*}(x',s',t)$. (Note that $b_{\langle x,s\rangle,j}$ for $j<\langle x,s\rangle$ is undefined as they are never enumerated into $B$.)

    For a fixed pair $x,s$, since $\lim_{t}g^{*}(x,s,t)$ is assumed to exist, then there must be some finite stage $t_{x,s}$ such that for all $t\geq t_{x,s},\,g^{*}(x,s,t_{x,s})=g^{*}(x,s,t)$. This implies that for all $t>\max\{t_{x',s'}\mid\langle x',s'\rangle\leq\langle x,s\rangle\}$, we always define $b^{l}_{\langle x,s\rangle,t+1}=b^{l}_{\langle x,s\rangle,t}$ and $f_{B}^{-1}(b^{l}_{\langle x,s\rangle,t})\uparrow$. Thus, there is some finite stage $t$ after which the left-most element in the collection $\{b_{\langle x,s\rangle,j}\mid j\geq\langle x,s\rangle\}$ never again changes. That is, the collection forms an $\omega$ chain. From the argument above and an analysis of the actions in Step 2, it is easy to conclude that the index $j\in\omega$ for which $b_{\langle x,s\rangle,j}=b^{l}_{\langle x,s\rangle,t}$ is such that $j\geq t_{x',s'}$ for any $\langle x',s'\rangle\leq\langle x,s\rangle$. Therefore, (i) is true.

    \item For any $s\in\omega$, the collection $\{d_{s,j}\mid j\in\omega\}$ forms a $\zeta$ chain iff there exists some $x\leq s$ such that $P(x,s)$ fires at infinitely many stages.

    To see why (ii) holds, we turn our attention to Step 4 of the construction. First suppose that for all $x\leq s$, $P(x,s)$ fires only finitely often. In particular, there is some finite stage $s^{*}$ large enough such that after $s^{*}$, $P(x,s)$ never again fires for all $x\leq s$. An analysis of Step 4 in the construction will allow one to conclude that for all $s'\geq s^{*}$, $d^{l}_{s,s'}=d^{l}_{s,s^{*}}$, and that $f_{A}^{-1}(d^{l}_{s,s^{*}})\uparrow$. In other words, the collection $\{d_{s,j}\mid j\geq s\}$ forms an $\omega$ chain.

    On the other hand, if there is some $x\leq s$ for which $P(x,s)$ fires infinitely often, then let these stages be denoted by $s_{0}<s_{1}<s_{2}<\dots$. Obviously, at all such stages, we would have defined $d^{l}_{s,s_{i}}=f_{A}^{-1}(d^{l}_{s,s_{i}-1})$, meaning that the collection $\{d_{s,j}\mid j\in\omega\}$ has no left-most element. Furthermore, as described in Step 4, we always define $d^{r}_{s,s'}=f_{A}(d^{r}_{s,s'-1})$, thus also ensuring that $\{d_{s,j}\mid j\in\omega\}$ has no right-most element. Thus, $\{d_{s,j}\mid j\in\omega\}$ forms a $\zeta$ chain iff there is some $x\leq s$ for which $P(x,s)$ fires infinitely often.
\end{enumerate}
To see that $A$ and $B$ are of the correct isomorphism type, it suffices to show that they consist of infinitely many $\omega$ chains and infinitely many $\zeta$ chains, as the finite orbits are enumerated whenever they are observed to close in the given punctual copy. Step 3 of the construction ensures that $B$ consists of $\zeta$ chains of the form $\{c_{i,j}\mid j\in\omega\}$ for each $i\in\omega$. By applying (i), we may also conclude that $B$ contains infinitely many $\omega$ chains given by $\{b_{i,j}\mid j\geq i\}$ for each $i\in\omega$. In fact, it is not difficult to see that these are exactly all the $\omega$ chains in $B$.

In $A$, it is easy to see that there are infinitely many $\omega$ chains, given by $\{a_{i,j}\mid j\in\omega\}$ for each $i\in\omega$. Since $g(x)$ is a total $\Delta_{3}^{0}$ function, for each $x\in\omega$, there must be some $s_{x}\in\omega$ such that $P(x,s_{x})$ fires infinitely often. Applying (ii) from above, we may then obtain that the chains $\{d_{s_{x},j}\mid j\in\omega\}$ for each $x\in\omega$ are all $\zeta$ chains. Thus, $A$ contains both infinitely many $\omega$ chains and also infinitely many $\zeta$ chains. We also note here that as a consequence of (ii), the only $\zeta$ chains in $A$ are of the form $\{d_{s_{x},j}\mid j\in\omega\}$ for some $x\in\omega$.

Let $h:A\to B$ be an isomorphism. To recover $g(x)$, consider the following procedure. Compute $h^{-1}(c_{i,0})$ for each $i\leq x$. Since $h$ is an isomorphism, we obtain that $h^{-1}(c_{i,0})$ must also be contained in some $\zeta$ chain of $A$. Recall that the only $\zeta$ chains of $A$ are of the form $\{d_{s_{y},j}\mid j\in\omega\}$ for some $y\in\omega$. By pigeonhole principle, at least one of $h^{-1}(c_{i,0})=d_{s,j}$ for some $j$ and $s$ where $s\geq s_{x}$. Let $s^{*}$ be the maximum of all such $s$ obtained from $h^{-1}(c_{i,0})$. Clearly, $s^{*}$ has the property that $\lim_{t}g^{*}(x,s^{*},t)=\lim_{t}g^{*}(x,s',t)$ for any $s'\geq s^{*}$. Or equivalently, $\lim_{t}g^{*}(x,s^{*},t)=\lim_{s}\lim_{t}g^{*}(x,s,t)$.

Given $s^{*}$ from above, compute $h(a_{i,0})$ for all $i\leq\langle x,s^{*}\rangle$. Each of these images must be the left-most elements of $\omega$ chains in $B$. Since $B$ only contains $\omega$ chains of the form $\{b_{\langle y,s\rangle,j}\mid j\geq\langle y,s\rangle\}$, we have that each $h(a_{i,0})=b_{\langle y,s\rangle,t}$ for some $y,s,t\in\omega$. Pick $t^{*}$ with the following properties.
    \begin{itemize}
        \item $h(a_{i^{*},0})=b_{\langle y,s\rangle,t^{*}}$ for some $i^{*}\leq\langle x,s^{*}\rangle$, and
        \item for any other $b_{\langle y',s'\rangle,t'}=h(a_{i,0})$, $\langle y',s'\rangle\leq\langle y,s\rangle$.
    \end{itemize}
    By pigeonhole principle, it must be that $\langle x,s^{*}\rangle\leq\langle y,s\rangle$, and by (i), $t^{*}$ has the property that for any $t\geq t^{*},\,g^{*}(x,s^{*},t^{*})=g^{*}(x,s^{*},t)$. In conclusion, $g^{*}(x,s^{*},t^{*})=\lim_{t}g^{*}(x,s^{*},t)=\lim_{s}\lim_{t}g^{*}(x,s,t)$. Since the procedure described above is clearly primitive recursive, $g\leq_{pr}h\oplus h^{-1}$, and thus, $\prcat(A)\subseteq\cone(\Delta_{3}^{0})$. Theorem~\ref{thm:d3inj} is proved.
\end{proof}

\section{A pathological injection structure}\label{sec:pathological}

\begin{theorem}\label{thm:inj}
    There exists a structure $A$, computably categorical but not punctually categorical, such that $\prcat(A)\nsubseteq\cone(\Delta_{1}^{0})$.
\end{theorem}

We shall construct our structure to be a punctual injection structure. Recall from Section \ref{sec:inj} that this injection structure necessarily has no infinite orbits, and no orbit sizes repeated infinitely often. The intuitive idea for the proof would be to use some pseudo-`pressing' strategy to ensure that the `opponents' copy our structure as much as possible. Where such `pressing' is impossible, we instead rely on our function oracle to compute the isomorphism.

\begin{proof}
Fix some computable but not primitive recursive function $g$, let $\{\Psi_{e}\}_{e\in\omega}$ be a list of all primitive recursive schemes, and let $\{B_{e}\}_{e\in\omega}$ be a list of all punctual structures in the language of one unary function symbol. We construct $d$, a $\Delta_{2}^{0}$ function oracle, and a punctual injection structure $A=(\omega,f_{A})$, satisfying the following requirements. 
\begin{align*}
    P_{e}:&\,\Psi_{e}^{d}\neq g.\\
    Q_{e}:&\,d\geq_{pr}q_{e}\oplus q_{e}^{-1}\text{ where }q_{e}:A\to B_{e}\text{ is an isomorphism, or }A\not\cong B_{e}.
\end{align*}

\

\emph{Strategy for one $Q_{e}$ requirement:} Fix some punctual structure $B_{e}$ with a single unary function. If such a function is ever discovered to not be injective, then $Q_{e}$ is trivially satisfied. We may thus assume that $B_{e}$ is an injection structure.

When the construction begins, we enumerate only finite orbits of size $n$. Although $B_{e}$ is punctual and should enumerate its elements `quickly', we note that arbitrary $B_{e}$ could still be much `slower' than the structure $A$ defined by us. The intention here is obviously to use $d$ to encode this delay. More specifically, since we only enumerate finite orbits of size $n$, if $B_{e}$ ever shows some orbit not of size $n$ (or more generally if $B_{e}$ is not a substructure of $A$), then we will never enumerate such a size into our structure $A$, ensuring that $A\not\cong B_{e}$. Thus, $B_{e}$ should eventually enumerate an orbit of size exactly $n$. Once it does so, we define $d(0)$ to be the stage number at which this happens. The intention here is that $d(j)$ encodes the stage number at which $B_{e}$ enumerates the $j^{th}$ orbit enumerated into $A$.

\

\emph{Strategy for one $Q_{e}$ and infinitely many $P_{e}$:} Fix some infinite sequence of sizes $n_{0}<n_{1}<n_{2},\dots$, and begin the strategy for $Q_{e}$ as described previously. Enumerate only orbits of size $n_{0}$ and define $d(j)$ to be the stages at which $B_{e}$ enumerates a suitable $q_{e}$-image for the $j^{th}$ orbit in $A$. Repeat this process until $\Psi_{0}^{d}\neq g$ is witnessed. Since $g$ is not primitive recursive and $d$ currently contains only primitive recursive information, $\Psi_{0}^{d}\neq g$ must be witnessed at some finite stage $s_{0}$. Once such a stage is reached, we define $d(s_{0})[s_{0}]=0$. The idea is that $d(s_{0})$ should encode a stage large enough such that $B_{e}$ contains $A[s_{0}]$ as a substructure. As long as $A[s_{0}]$ is not a substructure of $B_{e}$, $d(s_{0})$ remains at $0$.

Suppose that for all $j<i$, we have defined $s_{j}$, and $\Psi_{j}^{d}\neq g$ is witnessed by the initial segment $d\restriction s_{j}$. We now begin enumerating only orbits of size $n_{i}$ into our structure $A$. Just as before, if it is discovered that $B_{e}$ is not a substructure of $A$, then $Q_{e}$ will be satisfied by keeping $A\not\cong B_{e}$. Similarly, $d(s_{i-1}+k)$ will encode the stage at which the $k^{th}$ orbit of size $n_{i}$ is enumerated into $B_{e}$. Continue this process until it is discovered that $\Psi_{i}^{d}\neq g$. Since $d$ consists of only primitive recursive information (with finitely many exceptions, i.e., $d(s_{j})$ for each $j<i$) and $g$ is not primitive recursive, $\Psi_{i}^{d}\neq g$ must be witnessed at some finite stage $s_{i}$. We may then recursively repeat this procedure as described above, defining $d(s_{i})[s_{i}]=0$ and enumerating orbits of size $n_{i+1}$ into $A$.

However, for $i>0$, the failure of $\Psi_{i}^{d}$ to compute $g$ is witnessed by the initial segment $d\restriction s_{i}$. In particular, this initial segment includes $d(s_{j})$ for $j<i$. At stages $s$ where $A[s_{j}]$ is discovered to be a substructure of $B_{e}$, the value of $d(s_{j})$ changes from $0$ to $s$, possibly changing the computation in such a way that $\Psi_{i}^{d}$ now looks to be equal to $g$. When this happens, we initialise all requirements $P_{k}$ for $k>j$ and return to enumerating only orbits of size $n_{j+1}$ and restarting our definition of $s_{k}$ for all $k>j$. Clearly, each $P_{i}$ can only be initialised if $d(s_{j})$ changes for some $j<i$, and thus each $P_{i}$ will be initialised only finitely often.

\

\emph{The general strategy:} We order the requirements $Q_{0},P_{0},Q_{1},P_{1},\dots$. During the construction, we maintain $n_{0}<n_{1}<\dots$ and $s_{0}<s_{1}<\dots$ as distinct sets of \emph{markers} with the usual rules; whenever $n_{i}$ or $s_{i}$ changes, all $n_{j}$ or $s_{j}$ for $j>i$ will be \emph{kicked} to some value $>n_{i}$ or $>s_{i}$ respectively. As described earlier, we begin by enumerating orbits of size $n_{0}$ into $A$ and record the stages at which these orbits are enumerated into $B_{0}$. At each stage $s$ of the construction, as long as $\Psi_{0}^{d\restriction s}=g$, we kick $s_{i}$ for all $i\geq 0$ and define $s_{0}=s+1$. Once $\Psi_{0}^{d\restriction s}\neq g$ is witnessed, let $s_{0}=s$ and define $d(s_{0})=0$.

Suppose now that we have currently witnessed $\Psi_{j}^{d\restriction s_{j}}\neq g$ for each $j<i$. We will attend to $Q_{0},Q_{1},\dots,Q_{i}$ and $P_{i}$ in subsequent stages. Similar to before, we enumerate only orbits of size $n_{i}$ while waiting for the stage $s$ such that $\Psi_{i}^{d\restriction s}\neq g$. As long as $\Psi_{i}^{d\restriction s}=g$, we kick $s_{k}$ for all $k\geq i$ at stage $s$. For inputs $x>s_{i-1}$, we use $d(x)$ to record the stage at which the $x-s_{i-1}$-th orbit of size $n_{i}$ is enumerated into the structures $B_{0},B_{1},\dots,B_{i}$. More specifically, we define $d(s_{i-1}+k)$ to be an $(i+1)$-tuple, where each entry is the stage at which $B_{j}$ enumerates the $k^{th}$ ($>0$) orbit of size $n_{i}$ for each $j\leq i$. We continue this process until the first stage $s$ where one of the following happens.
\begin{itemize}
    \item If $\Psi_{i}^{d\restriction s}\neq g$, then define $d(s_{i})[s]$ to be the all zero $i$-tuple, and proceed to attend to $Q_{0},Q_{1},\dots,Q_{i+1}$, and begin the process to find $s_{i+1}$.
    
    \item For some $j<i$, $A[s_{j}]$ is discovered to be a substructure of $B_{k}$ for some $k\leq j$ at stage $s$. At such a stage $s$, change the $k^{th}$ entry of $d(s_{j})$ to $s$ (recall that $d(s_{j})$ is a $j$-tuple), define $s_{j+1}=s+1$ and initialise the requirements $P_{j'},Q_{j'}$ for all $j'>j$. Then return to attending to $Q_{0},Q_{1},\dots,Q_{j+1}$ and restart the process to find $s_{j+1}$.

    \item There is an additional point to note when we have multiple $Q_{e}$ requirements. Recall that for a single $Q_{e}$, if for some $k$, $B_{e}$ reveals more orbits of size $k$ than $A$ currently contains, then we are able to permanently keep $A\not\cong B_{e}$. However, the different $Q_{e}$ might possibly interact in a way that `damages' our strategy of recovering an isomorphism from $d$.
    
    In particular, the information we have encoded into $d$ is not quite sufficient; we need to be able to compute where each structure $B_{e}$ places the orbits of size $k$ for $n_{j}<k<n_{j+1}$ (which might have been enumerated into $A$ before $k$ is restrained from $A$). Thus, we shall also track the stages $t_{e}$ at which we `act' for $Q_{e}$ (see the construction for the definition). If we never do so, then let $t_{e}\uparrow$. Otherwise, $d(s_{j})$ will also encode the stage at which $A[t_{j+1}]$ becomes a substructure of $B_{k}$ for various $k\leq j$.
\end{itemize}
Intuitively, the first case is when we manage to successfully define $s_{i}$ without any changes to $d\restriction s_{i-1}$. The latter cases are where we make a change in $d\restriction s_{i-1}$ before finding $s_{i}$. When this happens, pick the maximum $s_{j}$ for which $d\restriction s_{j}$ did not change, and return to the strategy for attending to the requirements with indices $\leq j+1$. We may then restart the definition of $d(x)$ for $x>s_{j}$. Recall that $d(x)$ on all such $x$ was defined for the purpose of the strategy of $Q_{j+1}$ anyways, and can thus be initialised.

\

\emph{Construction:} Each stage of the construction could possibly last for a non-primitive recursive amount of time. During the construction, we also track sub-stages; stages of primitive recursive time at which we enumerate new elements into $A$.

At stage $0$, define $n_{i}[0]=i+1$ and $s_{i}=i$ for each $i\in\omega$. Begin enumerating orbits of size $n_{0}[0]$ into $A$, and do the following.
\begin{enumerate}[label=(\roman*)]
    \item If $B_{0}$ remains a substructure of $A$, then within primitive recursive time, $B_{0}$ must enumerate the first orbit of size $n_{0}[0]$. When it does so, we record the sub-stage at which this happens using $d(0)$.
    
    \item If $B_{0}$ is not a substructure of $A$, then define $d(0)=0$. Since $n_{0}$ is currently $1$, and $A$ only contains orbits of size $n_{0}[0]$, in order for $B_{0}$ to not be a substructure of $A$, one of the following must happen.
    \begin{itemize}
        \item $B_{0}$ enumerates an orbit which has yet to close. By ensuring that our structure has no infinite orbits, $B_{0}$ can only be isomorphic to $A$ if all orbits of $B_{0}$ are also finite. Thus, any orbit enumerated by $B_{0}$ has to eventually close. That is, we need not consider orbits which do not close.

        \item $B_{0}$ enumerates a (closed) orbit of size $k>n_{0}[0]$. Then define $n_{0}[1]$ to be some number $>k$, kicking all subsequent $n_{j}$ for $j>0$. Throughout the construction, we shall only enumerate orbits of size $n_{i}[s]$ for some $i$ at each stage $s$ of the construction. This ensures that we shall never enumerate any orbit of size $k$ into $A$, as $k<n_{0}[1]\leq n_{0}[s]<n_{1}[s]<\dots$ for all $s>0$.

        \item $B_{0}$ enumerates more orbits of size $n_{0}$ than $A$. Then we similarly define $n_{0}[1]$ to be a fresh number. In subsequent stages of the construction, we will never enumerate any orbit of size $n_{0}[0]$.
    \end{itemize}
    If $n_{0}$ is defined to be some fresh number by an action above, then we say to have \emph{acted} for $Q_{0}$, and define $t_{0}$ to be the sub-stage at which we did so. 
\end{enumerate}
Once we have succeeded in defining $d(0)$, proceed to stage 1.

At stage $s$, we say that $Q_{i}$ or $P_{i}$ \emph{requires attention} if
\begin{itemize}
    \item we have never \emph{acted} for $Q_{i}$ (to be defined later), and there is some orbit of size $k\geq n_{i}[s]$, such that $B_{i}$ contains more orbits of size $k$ than $A$, or
    \item $\Psi_{i}^{d\restriction s}$ is currently equal to $g$ respectively.
\end{itemize}
In addition, if we have acted for $Q_{i}$, then let $t_{i}$ be the sub-stage at which we did so. At the beginning of stage $s$, for any $i\geq 0$ and $j\leq i$ such that $s_{i}<s$ and we have yet to act for $Q_{j}$, do the following.
\begin{itemize}
    \item If $t_{i+1}$ is defined, and the structures $A[s_{i}]$ and $A[t_{i+1}]$ are both substructures of $B_{j}$, then define the $j^{th}$ entry of $d(s_{i})$ to be the sub-stage at which this was first witnessed. Otherwise, let the $j^{th}$ entry of $d(x)$ for all $x\geq s_{i}$ be $0$.

    \item If $t_{i+1}$ is undefined, then we only check whether $A[s_{i}]$ is a substructure of $B_{j}$. If it is, then define the $j^{th}$ entry of $d(s_{i})$ to be the first sub-stage at which it was witnessed. Otherwise, define the $j^{th}$ entry of $d(x)$ for all $x\geq s_{i}$ to be $0$. 
\end{itemize}
Once all such changes are made, let $i$ be the least for which $Q_{i}$ or $P_{i}$ requires attention.

If it was $Q_{i}$ which required attention, then define $n_{i}[s+1]$ to be some fresh number, and proceed to stage $s+1$. We then say that we have \emph{acted} for $Q_{i}$ and define $t_{i}$ to be the current sub-stage.

If it was $P_{i}$ which required attention, then define $s_{i}=s+1$ and define $d(x)$ for all $s_{i-1}<x<s_{i}=s+1$ as follows. While attempting to define $d(x)$, we enumerate only orbits of size $n_{i}[s]$ into $A$. Let $j\leq i$ be given.
\begin{enumerate}[label=(\roman*)]
    \item Define the $j^{th}$ entry of $d(x)$ to be $0$ if one of the following holds.
    \begin{itemize}
        \item $B_{j}$ is currently not a substructure of $A$. To be more specific, this means that there is some finite size $k$ such that $B_{j}$ currently has more orbits of size $k$ than $A$ does.
        
        \item The $j^{th}$ entry of $d(y)$ is currently $0$ for some $y<x$. Intuitively, this means that $B_{j}$ currently does not look isomorphic to $A$.
        
        \item We have acted for $Q_{j}$ at some earlier stage. In fact, once we have acted for $Q_{j}$, let the $j^{th}$ entry of $d(x)$ for all $x\geq s_{j}$ be $0$.
    \end{itemize}
    In the following cases, we may thus assume that none of the above hold.

    \item Define the $j^{th}$ entry of $d(x)$ to be the sub-stage at which the $x-s_{i-1}$-th orbit of size $n_{i}[s]$ is enumerated into $B_{j}$. Since $B_{j}$ is punctual, it has to enumerate new elements primitive recursively. In addition, new orbits either grow in size quickly or close and reveal its final size. Applying this with the fact that there is only a fixed number of orbits of size $<n_{i}[s]$ currently in $A$, and we are only enumerating orbits of size $n_{i}[s]$ into $A$, after some primitive recursive delay, one of the following must happen.
    \begin{itemize}
        \item $B_{j}$ enumerates the $x-s_{i-1}$-th orbit of size $n_{i}[s]$. As mentioned earlier, we record the sub-stage at which this happens in the $j^{th}$ entry of $d(x)$.

        \item There is some $k\geq n_{j}[s]$ such that $B_{j}$ has more orbits of size $k$ than $A$ does. That is, $Q_{j}$ would require attention at the next stage, and if acted for, we claim that we never again enumerate any orbits of size $k$ (to be verified later).
        
        \item There is some $k<n_{j}[s]$ such that $B_{j}$ has more orbits of size $k$ than $A$ does. In this case, we cannot forcefully kick higher priority $n_{j'}[s]\leq k$ for the sake of $Q_{j}$. We instead ensure that each size appears only finitely often in $A$, meaning that we can non-uniformly fix $q_{j}^{-1}:B_{j}\to A$ on all orbits of size $<n_{j}$, allowing us to satisfy $Q_{j}$ via computing $q_{j}\oplus q_{j}^{-1}$ from $d$ (verified below). In the meantime, define the $j^{th}$ entry of $d(x)$ to be $0$.
    \end{itemize}
\end{enumerate}
Once we have successfully defined $d\restriction(s+1)$, then proceed to stage $s+1$.

\

\emph{Verification:} We first make some preliminary observations. Notice that each $n_{e}$ only changes if $n_{e'}$ for some $e'<e$ changes, or we acted for $Q_{e}$ during the construction. Since we only act at most once for each $Q_{e}$, it is easy to see that each $n_{e}$ stabilises eventually.

To see that $d$ is $\Delta_{2}^{0}$, observe that for each $x\in\omega$, and for each $s_{i}$, exactly one of $x>s_{i}$, or $x\leq s_{i}$ is true for cofinitely many stages of the construction. Note that this is simply a consequence of how each $s_{i}$ only increases when it is changed, and that we do not claim or require that $s_{i}$ stabilises here. For a given $x$, since $x=s_{x}[0]$, there exists a least $i$ for which $x\leq s_{i}$ for cofinitely many stages. As a result, $d(x)$ must be an $(i+1)$-tuple. For such an $i$, there is a stage $s^{*}$ such that $n_{i}$ never again changes after stage $s^{*}$. We may further assume inductively that $s^{*}$ is also such that $d\restriction x$ never changes after $s^{*}$. We argue that each entry of $d(x)$ stabilises. Fix some $j\leq i$ and consider the $j^{th}$ entry of $d(x)$.
\begin{itemize}
    \item If there is some $y<x$ such that the $j^{th}$ entry of $d(y)$ is $0$. Then the $j^{th}$ entry of $d(x)$ will also be defined to be $0$. Furthermore, by the assumption that $d\restriction x$ never again changes after stage $s^{*}$, the $j^{th}$ entry of $d(x)$ will forever remain at $0$.

    \item We may thus suppose that the $j^{th}$ entry of $d(y)$ is always $>0$ for each $y<x$. We further split into the cases where either $x<s_{i}$ or $x=s_{i}$. For $x<s_{i}$, recall that the $j^{th}$ entry of $d(x)[s]$ either records the sub-stage at which $B_{j}$ enumerates the $x-s_{i-1}$-th orbit of size $n_{i}[s]$, or records $0$ if $B_{j}$ was discovered to not be a substructure of $A$ at stage $s$. This means that there is some $k$ such that $B_{j}[s]$ contains more orbits of size $k$ than $A[s]$ does. By the assumption that we never act for $Q_{j}$ during the construction, it must be that $k<n_{j}$, otherwise $Q_{j}$ would require attention and be acted for, changing $n_{j}$ at some stage $s>s^{*}$ a contradiction. On the other hand, we only enumerate orbits of size $\geq n_{i}\geq n_{j}>k$ into $A$ after stage $s^{*}$; if $B_{j}$ ever enumerates more orbits of size $k$ than $A$ at stage $s>s^{*}$, it cannot possibly later become a substructure of $A$. In conclusion, the $j^{th}$ entry of $d(x)$ remains forever $0$ if $B_{j}$ is discovered to not be a substructure of $A$ at stage $s>s^{*}$.
    
    For $x=s_{i}$, the $j^{th}$ entry of $d(x)$ records one of the following.
    \begin{itemize}
        \item The sub-stage at which $A[s_{i-1}]$ becomes a substructure of $B_{j}$, or
        \item the sub-stage at which $A[t_{i}]$ becomes a substructure of $B_{j}$, provided $t_{i}$ is defined and larger than $s_{i-1}$, or
        \item $0$ otherwise.
    \end{itemize}
    It is easy to see that the first two points are c.e.~ events once $s_{i-1}$ has stabilised and $t_{i}$ has been defined if ever. Thus, $d(x)$ never changes after some finite stage.
\end{itemize}
In conclusion, we have that each $n_{e}$ eventually stabilises, and that $d$ is a $\Delta_{2}^{0}$ function. We are now ready to prove that $P_{e}$ and $Q_{e}$ are satisfied.

\

\emph{Each $P_{e}$ is satisfied:} We proceed by induction on $e\in\omega$ and show that $s_{e}$ eventually stabilises at some finite value. Let $e$ be the least such that $s_{e}$ changes infinitely often. From the construction, $s_{e}$ only changes if either $Q_{e}$ or $P_{e}$ requires attention or some $s_{e'}$ where $e'<e$ changes. Since $e$ is assumed to be the least for which $s_{e}$ changes infinitely, then we may assume that either $Q_{e}$ or $P_{e}$ requires attention infinitely often.
\begin{itemize}
    \item Recall that $Q_{e}$ can only require attention at stage $s$ if we have never acted for it. In particular, since $s_{e}$ is the least which changes infinitely often, $Q_{e}$ must eventually be the highest priority requirement which requires attention. At such a stage, we would have acted for $Q_{e}$, and it will never again require attention in the construction. We may thus suppose that it is $P_{e}$ which requires attention infinitely often.

    \item Suppose that $P_{e}$ requires attention infinitely often. Let $s^{*}$ be a stage large enough such that all of the following hold.
    \begin{itemize}
        \item $s_{0},s_{1},\dots,s_{e-1}$ and $n_{0},n_{1},\dots,n_{e}$ have all stabilised.
        \item $d(x)$ for all $x\leq s_{e-1}$ does not change after stage $s^{*}$.
        \item $P_{e}$ requires attention at stage $s^{*}$.
    \end{itemize}
    By our assumption that $e$ is the least for which $s_{e}$ changes infinitely and the earlier observations, such an $s^{*}$ must exist.

    If there is some stage $s>s^{*}$ at which $P_{e}$ does not require attention, then we claim that $P_{e}$ never again requires attention. Let $s$ denote the first stage after $s^{*}$ at which $P_{e}$ does not require attention. Similarly, let $s'$ denote the first stage after $s$ at which $P_{e}$ again requires attention. Recall that $P_{e}$ requires attention at some stage iff $\Psi_{e}$ temporarily computes $g$ with the current $d$. More specifically, we have that $\Psi_{e}^{(d\restriction s)[s])}\neq g$, $\Psi_{e}^{(d\restriction s')[s']}=g$, and at some stage $s''$ between $s$ and $s'$, $d(x)[s'']\neq d(x)[s''+1]$ for some $x<s$. It follows from the construction that $s_{e}[s]=s$, and by the choice of $s^{*}$, we obtain that $s_{e-1}<x<s_{e}[s'']$. For each such $x$, $d(x)$ records either $0$ in the $j^{th}$ entry or the sub-stage at which $B_{j}$ enumerates the $x-s_{e-1}$-th orbit of size $n_{e}$. Consider the following cases.
    \begin{itemize}
        \item If the $j^{th}$ entry of $d(x)[s]$ is nonzero, then it means that $B_{j}$ has enumerated the $x-s_{e-1}$-th orbit of size $n_{e}$. Since $B_{j}$ is punctual, this enumeration cannot be `undone', and the $j^{th}$ entry of $d(x)$ cannot change unless the $j^{th}$ entry of $d(y)$ changes to $0$ for some $y<x$. By choice of $s^{*}$ such a $y$ must be larger than $s_{e-1}$, as $d\restriction(s_{e-1}+1)$ is assumed to have stabilised. But the $j^{th}$ entry of $d(y)$ records the sub-stage at which $B_{j}$ enumerates the $y-s_{e-1}$-th orbit of size $n_{e}$. This means that the $j^{th}$ entry of $d(y)$ cannot possibly be $0$.
        
        \item As noted above, if there is some $y<x$ such that the $j^{th}$ entry of $d(y)$ is $0$, then the $j^{th}$ entry of $d(x)$ will similarly be $0$. We may thus assume that $x$ is the least for which the $j^{th}$ entry of $d(x)[s]$ is $0$. An analysis of the actions in the construction allows one to easily conclude that this can only happen if there is some $k\in\omega$ such that $B_{j}[s]$ contains more orbits of size $k$ than $A$. Obviously, $k<n_{j}\leq n_{e}$, otherwise, we would have acted for $Q_{j}$ at some stage after $s^{*}$, contradicting the assumption that $n_{j}$ has stabilised. Note that we could not have acted for $Q_{j}$ before $s^{*}$, otherwise the $j^{th}$ entry of $d(y)$ would be $0$ for all $y\geq s^{*}$ (by definition of $d(x)$).
    \end{itemize}
    In other words, $d(x)[s]$ once defined, never changes again for each $x<s$, implying that $g\neq\Psi_{e}^{(d\restriction s)[s]}=\Psi_{e}^{d\restriction s}$. That is, $P_{e}$ never again requires attention after stage $s$, contradicting the assumption that it requires attention infinitely often. Thus, we may assume that for all stages after $s^{*}$, $P_{e}$ always requires attention.

    In addition, recall that we may discover in primitive recursive time whether $B_{j}$ enumerates the $x-s_{e-1}$-th orbit of size $n_{e}$, or enumerates more orbits of size $k$ than $A$ does for some $k<n_{e}$. Thus, for each $x>s_{e-1}$, $d(x)$ can be computed primitive recursively. But this implies that $\Psi_{e}^{d}=g$, contradicting the assumption that $g$ is not primitive recursive. Therefore, $P_{e}$ cannot require attention for infinitely many stages.
\end{itemize}
It follows immediately that for each $e$, $P_{e}$ is satisfied as $\Psi_{e}^{d\restriction s_{e}}\neq g$, and thus $\Psi_{e}^{d}\neq g$.

\

\emph{Each $Q_{e}$ is satisfied:} During the construction, note that we enumerate orbits of size $n_{e}$ only at stages where $P_{e}$ is the highest priority requirement which requires attention. Using the observation that $s_{e}$ also stabilises eventually, it follows that after stage $s_{e}$, we never enumerate any orbit of size $k\leq n_{e}$ into $A$. Therefore, $A$ only consists of finitely many orbits of size $k$ for each $k\in\omega$.

It remains to argue that each $Q_{e}$ is satisfied. Since each $n_{e}$ stabilises at a finite value and we only ever enumerate orbits of size $n_{e}[s]$ for some $e,s$, $A$ contains no infinite orbits. We may thus assume that $B_{e}$ also contains no infinite orbits, as $B_{e}$ cannot possibly be isomorphic to $A$ otherwise. We split the proof into two cases as follows.

First, suppose that at some finite stage $s$, we acted for $Q_{e}$. Recall that this means that at stage $s$, $Q_{e}$ is the highest priority requirement which required attention. In particular, we have yet to act for $Q_{e}$ at stage $s$, and there is some orbit of size $k\geq n_{e}[s]$ such that $B_{e}$ contains more orbits of size $k$ than $A$. We claim that after we act for $Q_{e}$, no more orbits of size $k$ will be enumerated into $A$, thus preserving $B_{e}\not\cong A$. When we act for $Q_{e}$, we define $n_{e}$ to be some fresh number at stage $s$, and kick all subsequent $n_{e'}$ for $e'>e$. That is to say, $k<n_{e'}[s']$ for each $e'\geq e$ and $s'>s$. By the properties of the markers $n_{i}$, we also have that $n_{i}[s]<n_{e}[s]\leq k$ for each $i<e$, and thus $k\neq n_{i}[s]$. Therefore, if we never act for any $Q_{i}$ for $i<e$ after stage $s$, $B_{e}$ cannot possibly be isomorphic to $A$, as no more orbits of size $k$ are ever enumerated. On the other hand, if there is some stage $s'>s$ during which we act for $Q_{i}$ where $i<e$, $n_{i}$ is similarly defined to be a fresh number at such a stage $s'$. This once again ensures that $n_{i}[s'']>k$ for any $i\in\omega$ and any $s''>s'$. In conclusion, once we have acted for $Q_{e}$, we never again enumerate any orbit of size $k$ into $A$.

Now suppose that we never act for $Q_{e}$ during the construction. If $Q_{e}$ requires attention at infinitely many stages, then there must be some finite stage during which we act for it. Thus, there must be some stage $s$ after which $Q_{e}$ never again requires attention. In particular, after such a stage, we must have that $B_{e}$ is always a substructure of $A$; as long as we are able to define $q_{e}:A\to B_{e}$ on orbits enumerated into $A$ after such a stage, $q_{e}^{-1}$ is naturally also primitive recursive. We fix such a stage as $s^{*}$, and define $q_{e}:A\to B_{e}$ as follows. Let $k$ be the largest such that there is an orbit of size $k$ in $B_{e}$ at stage $s^{*}$. In addition, let $i\geq e$ be the least such that $n_{i}>k$. Observe that after stage $s_{i}$ in the construction, we never enumerate any orbit of size $<n_{i}$. In particular, by non-uniformly fixing $q_{e}$ on $A\restriction s_{i}$, we need only define $q_{e}$ on orbits of size $\geq n_{i}>k$.

Let $a_{0},a_{1},\dots$, be the orbits enumerated into $A$ after stage $s_{i}$. We further assume that for each stabilising value of $n_{j}$, we enumerate at least one orbit of size $n_{j}$ into $A$. In particular, for each $j\in\omega$, $a_{j}$ has size at most $n_{i+j}$. For convenience, we refer to $d(s_{j})$ for $j\in\omega$ as the \emph{special} entries of $d$, and call $d(x)$ for all other $x$ as the \emph{non-special} entries. Note that we do not truly need to differentiate between the two; this definition is made purely for ease of argument. Define the primitive recursive scheme computing $q_{e}$ from $d$ as follows. Given $j\in\omega$, compute $d(x)$ for each $x\geq s_{i}$ until we have computed $d(s_{i}+2(j+1)+1)$. Since each entry is either special or non-special, this process guarantees that we must either have computed $d$ on at least $j+2$ many special entries, or at least $j+2$ many non-special entries.

We shall argue that by the sub-stage given by the maximum of the $e^{th}$ entry of the various $d(x)$ computed by the procedure above, an appropriate $q_{e}$ image for $a_{j}$ must have been enumerated into $B_{e}$. Consider the cases as follows.
\begin{itemize}
    \item Suppose that we have computed at least $j+2$ many special entries of $d$. That is, we have computed $d$ on the entries $s_{i}<s_{i+1}<\dots<s_{i+j+1}\leq s_{i}+2(j+1)+1$. Since $a_{j}$ has size at most $n_{i+j}$, and all such orbits must have been enumerated into $A$ before stage $s_{i+j+1}$, then it follows that the $e^{th}$ entry of $d(s_{i+j+1})$ encodes a sub-stage large enough such that $A[s_{i+j+1}]$ has become a substructure of $B_{e}$, providing an appropriate $q_{e}$ image for $a_{j}$.

    \item Suppose that the previous case does not hold. That is, there exists $k\leq j$ such that $s_{i+k}\leq s_{i}+2(j+1)+1<s_{i+k+1}$. From the assumption earlier that there is at least one orbit of size $n_{l}$ for each $l\in\omega$, $a_{j}$ must have size $<n_{i+k+1}$. If $a_{j}$ has size $<n_{i+k}$, then using the same argument as before, the sub-stage at which a $q_{e}$ image for $a_{j}$ is enumerated into $B_{e}$ happens before the sub-stage given by the $e^{th}$ entry of $d(s_{i+k})$. In addition, if $a_{j}$ has size $>n_{i+k}$, then $t_{i+k+1}\downarrow$ and $a_{j}$ must have been enumerated into $A$ by stage $t_{i+k+1}$. Recall that if $t_{i+k+1}\downarrow$, then the $e^{th}$ entry of $d(s_{i+k})$ is at least as big as the sub-stage at which $A[t_{i+k+1}]$ becomes a substructure of $B_{e}$. Therefore, we may find an appropriate $q_{e}$ image for $a_{j}$ by such a sub-stage.

    We may thus assume that $a_{j}$ has size exactly $n_{i+k}$. Applying the case assumption, we also obtain that we have computed $d$ on at least $j+2$ many non-special entries, denoted by $x_{0}<x_{1}<\dots<x_{j+1}$. Recall that non-special entries of $d(x)$ records in the $e^{th}$ entry, the sub-stage at which $B_{e}$ enumerates the $x-s_{l}$-th orbit of size $n_{l}$, where $l$ is the largest such that $s_{l}<x$. Therefore, each of the $e^{th}$ entries of $d(x_{0}),d(x_{1}),\dots,d(x_{j+1})$, should encode the sub-stage at which some orbit of size $n_{l}$ enters $B_{e}$ for some $l\geq i$. Since $a_{j}$ is the $j+1$-th orbit of size $\geq n_{i}$ to be enumerated into $A$, it follows that at least one of the $e^{th}$ entries of $d(x_{0}),d(x_{1}),\dots,d(x_{j+1})$, gives a sub-stage large enough at which an appropriate $q_{e}$ image for $a_{j}$ must have been enumerated into $B_{e}$.
\end{itemize}
Thus, provided that $B_{e}\cong A$, that is, the $e^{th}$ entry of $d(x)$ is nonzero for all $x\geq e$, we may primitive recursively compute an isomorphism $q_{e}:A\to B_{e}$ and its inverse using $d$. Theorem~\ref{thm:inj} is proved.
\end{proof}

\section{Degrees of punctual categoricity}

In this section, we show that in every non-zero c.e.~Turing degree, there exists a PR-degree that is low for punctual isomorphism and a PR-degree that is a degree of punctual categoricity. We split the proof into Theorems \ref{thm:low} and \ref{thm:deg}.

\begin{theorem}\label{thm:low}
    Let $\mathbf{d}$ be a non-zero c.e.~Turing degree, then there exists $f\equiv_{T}\mathbf{d}$ such that $f$ is low for punctual isomorphism.
\end{theorem}

\begin{proof}
Let $W$ be a non-recursive c.e.~set and let $\{\Psi_{i}\}_{i\in\omega}$ and $\{A_{n}\}_{n\in\omega}$ be listings of all primitive recursive schemes and all punctual structures respectively. We construct a total function $f:\omega\to\{0,1\}$ such that $f\equiv_{T}W$ and satisfying the following requirements.
\begin{align*}
    R_{i,j,m,n}:&\text{ If }\Psi_{i}^{f}:A_{m}\to A_{n}\text{ and }\Psi_{j}^{f}:A_{n}\to A_{m}\text{ are isomorphisms such that }\Psi_{i}^{f}\circ\Psi_{j}^{f}\text{ is the}\\
    &\text{ identity map, then }A_{m}\text{ and }A_{n}\text{ are punctually isomorphic.}
\end{align*}

\

\emph{An initial description:} To construct $f\equiv_{T}W$, we use a standard permitting technique. Throughout the construction, there will be a set of \emph{markers} $\{m_{0}<m_{1}<\dots\}\subseteq\omega$ such that $f(m_{i})=W(i)$, and a marker $m_{i}$ may be moved during stages at which some $j<i$ is enumerated into $W$. This ensures that $W\leq_{T}f$ (to be proved later). To obtain $f\leq_{T}W$, we will only change $f(x)$ at stage $s$ if some $y\leq x$ is enumerated into $W$ at stage $s$.

We now turn our attention to the requirements $R_{i,j,m,n}$. Let $i,j,m,n\in\omega$ be given. The oracle $f$ begins as the all $0$ function. For a single requirement $R_{i,j,m,n}$, we maintain a \emph{pointer} $p$ which tracks the value at which $R_{i,j,m,n}$ is currently searching for an \emph{error}; let $f^{*}\coloneqq(f\restriction p)^{\frown}1^{\omega}$ and check if both of the following holds.
\begin{itemize}
    \item $\Psi_{i}^{f^{*}}\circ\Psi_{j}^{f^{*}}$ is the identity.

    \item $\Psi_{i}^{f^{*}}$ and $\Psi_{j}^{f^{*}}$ are both injective and preserve the signature of $A_{m}$ and $A_{n}$.
\end{itemize}
If no error is ever found at this value of $p$, then $\Psi_{i}^{f^{*}}$ must be a punctual isomorphism from $A_{m}$ to $A_{n}$, since $f^{*}$ is clearly primitive recursive. On the other hand, if one of the conditions do not hold, it must be witnessed at some finite stage $s$ with the oracle $f^{*}\restriction s$. By changing $f\restriction s$ to $f^{*}\restriction s$, and provided that $f\restriction s$ never again changes, this failure of one of the above conditions will be preserved, thus ensuring that $R_{i,j,m,n}$ is satisfied.

In order to stay consistent with the permitting strategy, \emph{permission} must be provided by $W$ for $f\restriction s$ to change. That is, some $x\leq p$ has to enter $W$ at some stage after the error for $R_{i,j,m,n}$ is found in order for $f(y)$, where $p\leq y<s$, to change. For a fixed $p$, permission might never be obtained. To circumvent this, $R_{i,j,m,n}$ should search for errors at multiple locations. More specifically, for the strategy of $R_{i,j,m,n}$, we also attempt to define a computable function $g$ as follows. At each stage $s$ of the construction, $p$ will `point' to some marker $m_{x}$. Once an error for $R_{i,j,m,n}$ has been found at (the current value of) $m_{x}$, define $g(x)$ to also be the current value of $W(x)$. After which we `point' $p$ to the marker $m_{x+1}$ and repeat the process. The function $g$ as defined is clearly computable, and provided that errors are always successfully found at each $m_{x}$, $g$ will also be total. Since $W$ is non-recursive, there must be some $x$ at which $x$ enters $W$ at a stage after an error was found for $R_{i,j,m,n}$ at $m_{x}$, thus allowing us to \emph{act} for $R_{i,j,m,n}$ by changing the oracle $f$.

As the strategy of each requirement $R_{i,j,m,n}$ only depends on whether an error can be found, we suppress the indices and simply refer to $R_{i,j,m,n}$ as $R_{e}$. Similarly, $p_{e}$ and $g_{e}$ will respectively be the pointer and computable function associated with the strategy for $R_{e}$.

\

\emph{The general strategy and injury:} The priority of the requirements will be arranged as $R_{0},R_{1},\dots$. Since the values of the markers are dynamic, when we say that $p_{e}$ \emph{points} to $m_{x}$, this means that the value of $p_{e}$ copies $m_{x}$ as it changes. At each stage of the construction, $R_{e}$ will always be searching for errors using $f_{e}\coloneqq (f\restriction p_{e})^{\frown}1^{\omega}$. Once an error is found, define $g_{e}(x)$ to be the current value of $W(x)$ where $x$ is such that $p_{e}$ currently points to $m_{x}$. Also let $p_{e}$ now point to $m_{x+1}$. For similar reasons as before, there is either some $x$ for which no error is ever found for $R_{e}$ at $x$ or $x$ eventually enters $W$, providing permission for $R_{e}$ to act.

While the strategy for $R_{e}$ is being pursued, one of the following could happen.
\begin{itemize}
    \item The value of $m_{x}$ at which $R_{e}$ previously found an error changes.

    \item The value of $m_{x}$ at which $p_{e}$ is currently pointing changes.

    \item The oracle changes value at $y$ for some $y<p_{e}$.
\end{itemize}
In the above scenarios, it is evident that any potential satisfaction found for $R_{e}$ may now become undone, and we say that $R_{e}$ is \emph{injured}. However, we shall show that in each of the above scenarios, progress in the strategy of a higher priority $R_{e'}$ must have been made. To give a brief description of why this is so, recall that $m_{x}$ changes only if some $x'<x$ enters $W$. In particular, $W$ has just provided permission to change $f(y)$ for any $y\geq x'$. That is, any $R_{e'}$ which had earlier found an error at $m_{x'}$ now has permission to act. As long as we arrange the construction such that such an $R_{e'}$ must be of higher priority or $e'=e$, it is clear that the potential injury is finite. To this end, throughout the construction, $p_{e}<p_{e'}$ if $e<e'$. In addition, for each $x$ such that $p_{e}<m_{x}<p_{e'}$, it must be that $R_{e'}$ has found an error at $m_{x}$.

\

\emph{Construction:} During the construction, each requirement $R_{e}$ will be declared either satisfied or unsatisfied. At the start of the construction, for each $e\in\omega$, let $m_{e}[0]=e$, point each $p_{e}$ at $m_{e}$, and declare $R_{e}$ to be unsatisfied. Also define $f(x)=0$ for all $x\in\omega$. At each stage $s$ of the construction, do the following.
\begin{description}
    \item[Step 1] Let $e$ be the least such that $R_{e}$ is currently unsatisfied. Then begin searching for an error for $R_{e}$ at $p_{e}$.
    
    \item[Step 2] Let $e$ be the least such that $R_{e}$ has found an error at $p_{e}$ and $R_{e}$ is currently unsatisfied. If no such $e$ exists, then proceed to Step 3. Otherwise, let $x$ be such that $p_{e}$ is currently pointing at $m_{x}$. Then define $g_{e}(x)=W(x)[s]$, and let $p_{e}$ now point to $m_{x+1}$. For each $e'>e$, declare $R_{e'}$ to be unsatisfied, define $g_{e'}(y)\uparrow$ for each $y\in\omega$, and point $p_{e'}$ to $m_{x+1+e'-e}$. For convenience, we refer to this action as \emph{initialising} $R_{e'}$.

    \item[Step 3] Let $e$ be the least such that $R_{e}$ has found an error at $m_{x}[s]$ for which $x$ has just entered $W$. Then define $f\restriction s$ to be $(f\restriction m_{x})^{\frown}1^{s-m_{x}}$, declare $R_{e}$ to be satisfied, and initialise $R_{e'}$ for each $e'>e$. Also define $m_{x'}[s+1]$ to be the least number $\geq s$ and $\neq m_{x''}[s+1]$ for any $x''<x'$. Notice that $f$ on all such inputs are currently $0$. For each $x'\geq x$ that is currently contained in $W$, define $f(m_{x'}[s+1])=1$.
\end{description}

\

\emph{Verification:} To verify that the construction works, we prove the following statements.
\begin{enumerate}[label=(\roman*)]    
    \item For each $x\in\omega$, the value of $m_{x}\coloneqq \lim_{s}m_{x}[s]$ never changes after the stage $s$ for which $W[s]\restriction x=W\restriction x$, and $f(m_{x})=W(x)$.
    
    \item For each $e\in\omega$, there is some $x_{e}\in\omega$ such that $p_{e}$ points at $m_{x_{e}}$ for cofinitely many stages.
\end{enumerate}
Let $x\in\omega$ be given. Fix the least stage $s$ such that $W[s]\restriction x=W\restriction x$. The only action in the construction which possibly changes the value of $m_{x}$ is in Step 3. In particular, if $m_{x}[s+1]\neq m_{x}[s]$, then at stage $s$, some $x'<x$ must have been enumerated into $W$. Since $s$ is a stage at which $W[s]\restriction x=W\restriction x$, this cannot happen after stage $s$, and thus $\lim_{s}m_{x}[s]=m_{x}[s+1]$. To see why $f(m_{x})=W(x)$, consider the following cases.
\begin{itemize}
    \item There is some stage $t$ such that $x$ enters $W$ at stage $t$. If $t\geq s$, then at stage $t$, Step 3 of the construction ensures that we define $f(m_{x}[t+1])=1$. Since the value of $m_{x}$ never changes after stage $s$, it follows that $f(m_{x})=f(m_{x}[t+1])=1$. On the other hand, if $t<s$, then it must be that $s>0$, and at stage $s$, some $x'<x$ enters $W$. Then the construction would have defined $f(m_{x}[s+1])=1$ since $x\in W[t]\subseteq W[s]$ and $x>x'$. Just as before, since the value of $m_{x}$ never again changes after stage $s+1$, $f(m_{x})=f(m_{x}[s+1])=1$. 

    \item If $x\notin W$, then we claim that $f(m_{x})=0$. Observe that whenever some $x'<x$ enters $W[t]$, $m_{x}[t+1]$ would be defined to be a fresh value; thus $f(m_{x}[t+1])[t]=0$. This implies that $f(m_{x})=f(m_{x}[s+1])[s]=0$ since no $x'<x$ ever enters $W$ after stage $s$.
\end{itemize}
Using (i), we may easily show that $f\equiv_{T}W$. Given $f$ and $x$, suppose that $m_{x'}$ has been computed for each $x'<x$, and search for the stage $s$ such that $W[s]\restriction x=\{x'\mid x'<x\text{ and }f(m_{x'})=1\}$. For such a stage $s$, it must be that $m_{x}[s+1]=m_{x}$, and we may then retrieve $W(x)$ by computing $f(m_{x})$. To compute $f(x)$, since $f(x)$ possibly changes only when some $y\leq x$ enters $W$, at the stage $s$ such that $W[s]\restriction (x+1)=W\restriction (x+1)$, $f(x)[s+1]=\lim_{s}f(x)[s]$.

From Steps 2 and 3 of the construction, it is evident that the markers $m_{x}$ at which each $p_{e}$ points at is increasing (in $x$). Thus, if $p_{e}$ does not point at a single marker for cofinitely many stages, then $p_{e}$ points at each marker for only finitely many (possibly $0$) stages. Let $e\in\omega$ be the least such that this holds. Following the construction, a pointer $p_{e}$ only changes the marker it points at due to one of the following reasons.
\begin{itemize}
    \item An error for $R_{e}$ was found at $m_{x}$ at which $p_{e}$ is currently pointing. Then $p_{e}$ subsequently points at $m_{x+1}$. (See Step 2 of the construction.)

    \item $R_{e}$ has acted (see Step 3).

    \item $R_{e}$ was initialised.
\end{itemize}
$R_{e}$ is only initialised if there is some $e'<e$ such that an error for $R_{e'}$ was found (see Step 2), or $R_{e'}$ has acted (see Step 3). In either of these situations, the marker $p_{e'}$ is pointing at changes. By the assumption that $e$ is the least for which the marker $p_{e}$ points at changes infinitely often, we may assume that pass some finite stage $s^{*}$, $p_{e}$ cannot change due to $R_{e}$ being initialised. After such a stage, $p_{e}$ changes only if it has acted or if an error is found at the current marker $p_{e}$ is pointing at.

If $R_{e}$ has acted, then it must have been declared to be satisfied. Furthermore, as no other higher priority requirement acts after $s^{*}$, Step 2 of the construction never again considers $R_{e}$. In other words, once $R_{e}$ acts after stage $s^{*}$, $p_{e}$ never again changes the marker it points at.

Finally, we consider the possibility that $R_{e}$ never acts after stage $s^{*}$. Then the only way that $p_{e}$ changes the marker it points at infinitely often is if errors are found at each marker that $p_{e}$ points at. In addition, $R_{e}$ never receives permission to act upon these errors; there is some $x$ such that for all $x'\geq x$, $g_{e}(x')\downarrow=W(x')$, contradicting the assumption that $W$ is non-recursive.

A quick analysis of the construction allows one to conclude that if $p_{e}$ stabilises, then $R_{e}$ must remain in the satisfied state for cofinitely many stages. In addition, this means that either some error for $R_{e}$ was found and acted upon, or there is some $x$ such that no error for $R_{e}$ is ever found at $m_{x}$. Thus, each $R_{e}$ is met. Theorem~\ref{thm:low} is proved.
\end{proof}

\

\begin{theorem}\label{thm:deg}
    Let $\mathbf{d}$ be a non-zero c.e.~Turing degree, then there exists $g\equiv_{T}\mathbf{d}$ such that $g$ is a degree of punctual categoricity.
\end{theorem}

\begin{proof}
Let $W$ be a non-recursive c.e.~set, $\{p_{e}\}_{e\in\omega}$ be a listing of primitive recursive functions, and let $\{B_{e}\}_{e\in\omega}$ be a listing of the punctual structures in the language of our structure $B$. We construct $g\equiv_{T}W$ satisfying the following requirements.
\begin{align*}
    G:&\text{ There exists punctual }B'\cong B\text{ such that if }f:B\to B'\text{ is an isomorphism,}\\
    &\text{then }f\oplus f^{-1}\geq_{pr}g.\\
    Q_{m,n}:&\text{ If }B_{m}\cong B_{n}\cong B,\text{ then there exists }h:B_{m}\to B_{n}\text{ an isomorphism, such that }g\geq_{pr}h\oplus h^{-1}.
\end{align*}
For each $x\in\omega$, define $g(\langle 0,x\rangle)$ to be the stage at which $x$ enters $W$, and $-1$ if no such stage exists. For each $e>0$ and each $x\in\omega$, also define $g(\langle e,x\rangle)$ to be the stage at which $p_{e-1}(x)\downarrow$. It is evident that $g\equiv_{T}W$.

\

\emph{An initial description of $B$:} We adopt the pressing strategy from \cite{kmn17}. The language of our structure will consist of four unary function symbols $S,P,R,C$ defined below. In the $e^{th}$ \emph{component} of our structure, we will have elements which form a \emph{spine}, given by $b_{e,i,0}$ for the various $j\in\omega$, and a distinguished element, $b_{e,0,0}$ which is the \emph{root} (illustrated in Fig. \ref{fig:deg}).

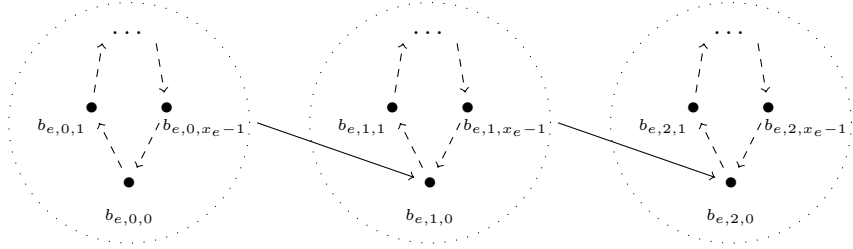
\begin{figure}
    \centering
    \begin{tikzpicture}
        \foreach \x in {0, 1, 2}
        {
        \node (\x0) at (4*\x,1) [label={[shift={(0,-0.9)}]\tiny $b_{e,\x,0}$}] {$\bullet$};

        \node (\x1) at (4*\x-0.5,2) [label={[shift={(-0.4,-0.7)}]\tiny $b_{e,\x,1}$}] {$\bullet$};
        \node (\x3) at (4*\x+0.5,2) [label={[shift={(0.5,-0.7)}]\tiny $b_{e,\x,x_{e}-1}$}] {$\bullet$};

        \node (\x2) at (4*\x,3) {$\dots$};

        \draw[->,dashed] (\x0) -- (\x1);
        \draw[->,dashed] (\x1) -- (\x2.south west);
        \draw[->,dashed] (\x2.south east) -- (\x3);
        \draw[->,dashed] (\x3) -- (\x0);
        \draw[-,loosely dotted,fill=none,draw=black] (4*\x,1.8) circle (1.6cm);
        }

        \draw[->] (1.7,1.8) -- (10);
        \draw[->] (5.7,1.8) -- (20);

    \end{tikzpicture}
    \caption{The $e^{th}$ component of $B$. The function arrows for $R$ and $P$ are not shown here. The solid arrows represent the function $S$, and the dashed arrows represent $C$.}
    \label{fig:deg}
\end{figure}

\begin{itemize}
    \item For each element along the spine, $b_{e,i,0}$, the function $C$ generates $x_{e}$-\emph{cycles}, consisting of finitely many elements $b_{e,i,j}$, for each $j<x_{e}$ with the property that $C(b_{e,i,x_{e}-1})=b_{e,i,0}$ and $C(b_{e,i,j})=b_{e,i,j+1}$ for each $j<x_{e}-1$.
    
    \item The range of the map $S$ consists only of elements along the spine. More specifically, for each $i$ and $j<x_{e}$, $S(b_{e,i,j})=b_{e,i+1,0}$.

    \item $P$ is the identity map on all elements of the form $b_{e,i,j}$.
    
    \item Finally, $R$ maps all elements within the $e^{th}$ component to $b_{e,0,0}$, the root. This ensures that in any potential copy of $B$, the roots of each component must be enumerated quickly after any element within that same component is enumerated.
\end{itemize}
To aid in the strategy for $G$ and the various $Q_{m,n}$, we shall keep the components `sufficiently' different by keeping the cycle sizes of each component different.

\

\emph{The pressing strategy:} Except for stages where we `switch' (see strategy for $R$), the structures $B$ and $B'$ will be enumerated in exactly the same way. In much the same spirit as the original pressing argument, we want to `force' the structures $B_{e}$ to remain as substructures of $B$ and $B'$ at each stage. In general, the $2e^{th}$ and $2e+1$-th components of $B$ will depend on the structures $B_{n}$ for each $n\leq e$. At each stage of the construction, only one component is not \emph{closed}: there is some element in the component such that $S$ is yet to be defined on that element.

For simplicity, first consider the $0^{th}$ component. Set the cycle size for the $0^{th}$ component as $1$. While waiting for $B_{0}$ to enumerate elements, continue enumerating the elements $b_{0,i,0}$ of the spine of the $0^{th}$ component and defining $C(b_{0,i,0})=b_{0,i,0}$. When some element is enumerated into $B_{0}$, compute $C^{m}R(b)$ for each $m\leq M_{1}$ (where $M_{1}\geq 1$ to be defined later).
\begin{itemize}
    \item If $CR(b)\downarrow=R(b)$, then it means that the structure $B_{0}$ has enumerated the root of the $0^{th}$ component.

     \item If the least $m>0$ for which $C^{m}R(b)\downarrow=R(b)$ is not $1$, then $B_{0}$ has revealed some cycle size not yet enumerated into $B$. By never enumerating such a cycle size into $B$, we may then guarantee that $B_{0}\not\cong B$. 

    \item If such an $m>0$ does not exist, that is, $C^{m}R(b)\downarrow\neq R(b)$ for each $0<m\leq M_{1}$, then we may intuitively think of this as the situation in which $B_{0}$ has revealed some `large' cycle size. While we may not directly obtain that $B_{0}\not\cong B$ in this case, by some careful arrangement of the cycle sizes of each subsequent component, we may avoid the eventual size of the cycle of $R(b)$.
\end{itemize}
Once one of the above is witnessed, we may then \emph{close} the $0^{th}$ component by defining $S(b_{0,i,0})=b_{0,i,0}$ where $i$ is the largest index such that $b_{0,i,0}$ currently exists in $B$. The idea here is that in the first case, we have successfully made $B_{0}$ `copy' our structure $B$, and hence may proceed to enumerating other components. In the second and third cases, we continue enumerating the $0^{th}$ component until some guarantee that $B_{0}\not\cong B$ is obtained before closing the $0^{th}$ component and starting new components. In particular, there is either some element $b\in B_{0}$ such that $R(b)$ has a different cycle size than anything enumerated into $B$ thus far, or has cycle size at least $M_{1}$.

Let $e>0$ be given and suppose that $M_{e}$ has been defined with the property there is at least one (non-zero) size $x_{e}<M_{e}$ that is available. In addition, for each $n\leq\lfloor e/2\rfloor$, one of the following holds.
\begin{itemize}
    \item For each $e'<e$, the root of the $e'$-th component has been enumerated into $B_{n}$; there is some $b\in B_{n}$ such that the least $m>0$ for which $C^{m}R(b)\downarrow=R(b)$ is such that $m=x_{e'}$.

    \item There is some $b\in B_{n}$ such that the least $m>0$ for which $C^{m}R(b)\downarrow=R(b)$ is such that $m\neq x_{e'}$ for any $e'<e$. Then such a size $m$ must have been declared unavailable, and we shall refer to such $B_{n}$ as \emph{inactive}.

    \item If $B_{n}$ is not inactive, but there is some $b\in B_{n}$ such that for each $0<m\leq M_{e},\,C^{m}R(b)\downarrow\neq R(b)$, then we say that $B_{n}$ is \emph{pending} and call $b$ the $e-1$-\emph{witness}.
\end{itemize}
The intuition is that for $B_{n}$ which are inactive, we already have that $B_{n}\not\cong B$, while the $B_{n}$ which are active are those that we still have to monitor during the construction. For those $B_{n}$ which are pending, there is still some `mystery' element $b\in B_{n}$ not revealed to be part of any component enumerated into $B$ thus far. The idea is that we do not want such an element to be able to mimic some component enumerated into $B$ at a later stage. In building the $e^{th}$ component, we need to take into account all such structures and ensure that we only enumerate cycles with sizes smaller than $R(b)$ for the various witnesses $b$.

Begin the $e^{th}$ component by enumerating the elements $b_{e,i,0}$ and letting the cycle sizes be $x_{e}$: for each $j<x_{e}-1$, let $C(b_{e,i,j})=b_{e,i,j+1}$ and $C(b_{e,i,x_{e}-1})=b_{e,i,0}$. We shall only close the $e^{th}$ component when, for each $n\leq\lfloor e/2\rfloor$ where $B_{n}$ is not inactive, either $B_{n}$ has enumerated the root of the $e^{th}$ component, or if $B_{n}$ has a $e$-witness $b$. Recall that this means that for each $0<m\leq M_{e+1}$ (sufficiently large, to be defined shortly), $C^{m}R(b)\downarrow\neq R(b)$. As before, while waiting for these conditions to be met, we continue extending the spine of the $e^{th}$ component to keep $B$ (and $B'$) punctual.
\begin{itemize}
    \item If $B_{n}$ was previously declared pending, then it has a $e-1$-witness $b$. That is to say, $R(b)$ cannot possibly have cycle size $x_{e}<M_{e}$, since $B_{n}$ has already shown that $C^{x_{e}}R(b)\downarrow\neq R(b)$. Continue computing $C^{m}R(b)$ for the $e-1$-witness $b$. This process either produces some $M_{e}< m\leq M_{e+1}$ such that $C^{m}R(b)\downarrow=R(b)$, or we find that for each $0<m\leq M_{e+1},\,C^{m}R(b)\downarrow\neq R(b)$. If the former holds, then declare $B_{n}$ to be inactive and declare the size $m$ to be unavailable (recall that this means that such a size is never enumerated after this stage). Furthermore, note that $m\neq x_{e'}$ for any $e'\leq e$; there are no cycles of size $m$ in $B$ currently. If the latter holds, then observe that $b$ is now a $e$-witness for $B_{n}$ (and trivially still a $e-1$-witness for $B_{n}$). In summary, for each pending $B_{n}$, it either remains pending, or becomes inactive.

    \item If $B_{n}$ is neither inactive nor currently pending, then for each element $b$ enumerated into $B_{n}$, compute for each $0<m\leq M_{e+1},\,C^{m}R(b)$. If for each $m\leq M_{e+1},\,C^{m}R(b)\downarrow\neq R(b)$, then declare $B_{n}$ to be pending with the $e$-witness $b$. We may thus assume that there is some $0<m\leq M_{e+1}$ for which $C^{m}R(b)\downarrow=R(b)$. Let $m$ be the least where this holds. If $m\neq x_{e'}$ (or $y_{e'}$, to be defined later) for any $e'\leq e$, then declare $m$ to be unavailable, and $B_{n}$ to be inactive. For the same reasons as before, $B_{n}$ cannot possibly be isomorphic to $B$. Recursively suppose that the $e'$-th component for each $e'<e$ is finite. That is, there are only finitely many elements $b'\in B$ such that $R(b')$ has cycle size $x_{e'}$ (or $y_{e'}$) for some $e'<e$. In particular, $B_{n}$ must eventually enumerate some element $b$ for which $R(b)$ does not have cycle size $x_{e'}$ for any $e'<e$, otherwise, we will again discover that $B_{n}\not\cong B$ (and can declare $B_{n}$ to be inactive). That is, if $B_{n}$ is not declared to be pending or inactive, then it must, in finite time, enumerate an element $b$ such that the least $m>0$ for which $C^{m}R(b)\downarrow=R(b)$ is exactly $x_{e}$.
\end{itemize}
Evidently, after starting the $e^{th}$ component in $B$, for each $n\leq\lfloor e/2\rfloor$, within finite time, each $B_{n}$ must either reveal a $e$-witness, become inactive, or enumerate the root of the $e^{th}$ component. Once such a stage is reached, then we close the $e^{th}$ component by defining $S(b_{e,i,j})=b_{e,i,0}$ for each $j<x_{e}$ and where $i$ is the unique index for which $S(b_{e,i,j})$ has yet to be defined. Roughly speaking, since each structure potentially only `blocks' one size, and each component uses only one size, it is not hard to see that as long as $M_{e+1}>(e+1)+(\lfloor e/2\rfloor+1)+1$, we may always successfully find $x_{e+1}<M_{e}$ that is available to be the cycle size of the $e+1$-th component. As shall be seen shortly there are some additional factors that contribute to the definition of $M_{e}$, thus we momentarily delay giving a definition.

\

\emph{The strategy for $G$:} The idea here is to use the image of the root of each component under a given isomorphism from $B$ to $B'$ to encode the values of $g$. However, since each component depends on a growing number of punctual structures, there cannot be a primitive recursive way to list the indices of the roots of each component. To overcome this, we shall dedicate the components with index $2e$ to encode a stage large enough such that $B$ has closed the $2e+2$-th component, and necessarily also closed the $2e+1$-th component. Roughly speaking, to obtain the indices of $b_{e',0,0}$ for each $e'\leq 2e$, compute the given isomorphism on $b_{0,0,0}$, obtaining the index of $b_{2,0,0}$, and repeating till we obtain the index of $b_{2e,0,0}$. Observe that this process is primitive recursive in the given isomorphism. The elements $b_{2e+1,0,0}$ will be used to encode the value of $g$.

We say that a component is \emph{ready} if the following holds.
\begin{itemize}
    \item If the component is the $2e^{th}$ component and the $2e+2$-th component has closed in the structure $B$ (and thus also $B'$).

    \item If the component is the $2\langle 0,x\rangle+1$-th component and $x$ has entered $W$.

    \item If the component is the $2\langle e+1,x\rangle+1$-th component and $p_{e}(x)$ has converged.
\end{itemize}

Suppose that all components in $B$ and $B'$ are currently closed. For the component with the lowest index, say $e$, that is ready, we perform the \emph{switch} as follows. Enumerate an exact duplicate of the $e^{th}$ component into $B$, containing the elements $d_{e,i,j}$ for each $i,j$ where $b_{e,i,j}$ is currently an element in the $e^{th}$ component. In particular, the collection of all the $b_{e,i,j}$ are indistinguishable from the collection of $d_{e,i,j}$ up to isomorphism. Do the same in $B'$. Note that by convention, each of these $d_{e,i,j}$, and also $d'_{e,i,j}\in B'$ has indices larger than the current stage number. Thus, we would like to `force' any isomorphism from $B\to B'$ to map the elements from the original component in $B$ to the duplicated component in $B'$. To achieve this, we enumerate a \emph{tail} to join the original $e^{th}$ component with the duplicate $e^{th}$ component (see Fig.~\ref{fig:degswitch}).
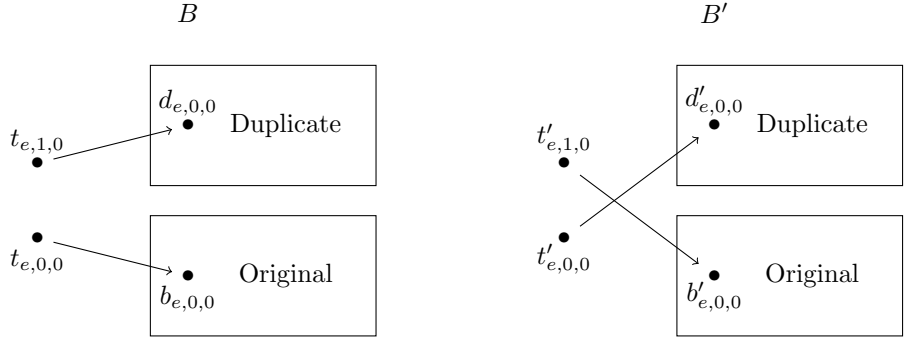
\begin{figure}
    \begin{tikzpicture}
        \node (b0) at (0,0) [label={[shift={(0,-0.8)}]$b_{e,0,0}$}] {$\bullet$};
        \node (d0) at (0,2) [label={[shift={(0,-0.2)}]$d_{e,0,0}$}] {$\bullet$};
        \node (t0) at (-2,0.5) [label={[shift={(0,-0.8)}]$t_{e,0,0}$}] {$\bullet$};
        \node (t1) at (-2,1.5) [label={[shift={(0,-0.2)}]$t_{e,1,0}$}] {$\bullet$};
        \node at (0,3.5) {$B$};
        
        \draw[-,fill=none] (-0.5,-0.8) rectangle (2.5,0.8);
        \node at (1.3,2) {Duplicate};
        \draw[-,fill=none] (-0.5,1.2) rectangle (2.5,2.8);
        \node at (1.3,0) {Original};

        \draw[->] (t0) -- (b0);
        \draw[->] (t1) -- (d0);

        \foreach \x in {7}
        {
        \node at (\x,3.5) {$B'$};
        
        \node (b'0) at (\x,0) [label={[shift={(0,-0.8)}]$b'_{e,0,0}$}] {$\bullet$};
        \node (d'0) at (\x,2) [label={[shift={(0,-0.2)}]$d'_{e,0,0}$}] {$\bullet$};
        \node (t'0) at (\x-2,0.5) [label={[shift={(0,-0.8)}]$t'_{e,0,0}$}] {$\bullet$};
        \node (t'1) at (\x-2,1.5) [label={[shift={(0,-0.2)}]$t'_{e,1,0}$}] {$\bullet$};

        \draw[-,fill=none] (\x-0.5,-0.8) rectangle (\x+2.5,0.8);
        \node at (\x+1.3,2) {Duplicate};
        \draw[-,fill=none] (\x-0.5,1.2) rectangle (\x+2.5,2.8);
        \node at (\x+1.3,0) {Original};
        }

        \draw[->] (t'0) -- (d'0);
        \draw[->] (t'1) -- (b'0);
    \end{tikzpicture}
    \caption{Switching in the $e^{th}$ component. The arrows represent the function $P$.}
    \label{fig:degswitch}
\end{figure}

Since the duplicate and original $e^{th}$ component are identical up to isomorphism, we have the freedom to attach the tail differently in $B$ and $B'$ as shown in the figure while maintaining $B\cong B'$. Evidently, $t_{e,0,0}$ and $t_{e,1,0}$ should be distinguishable up to isomorphism in order to ensure that isomorphisms from $B$ to $B'$ has to map $b_{e,0,0}$ to $d'_{e,0,0}$ (with large index) in $B'$. The tail will consist of elements labelled $t_{e,i,j}$ (in $B$) and the maps $S,P,R,C$ are defined on them as follows.
\begin{itemize}
    \item Let $e'$ be the largest index for which the $e'$-th component currently exists in $B$. Pick a size $0<y_{e}<M_{e'+1}$ different from any cycle size in $B$ thus far which is still available (we shall argue later that such sizes must exist by choosing a sufficiently large $M_{e'+1}$). All cycles of elements in the tail shall have size $y_{e}$.

    \item As shown in the figure define in $B$, $P(t_{e,0,0})=b_{e,0,0}$, $P(t_{e,1,0})=d_{e,0,0}$, and define in $B'$, $P(t'_{e,0,0})=d'_{e,0,0}$, and $P(t'_{e,1,0})=b'_{e,0,0}$. For the rest of the elements in the tail, $P$ is again the identity map. This will be the only place in which $B$ and $B'$ `locally' differ.

    \item $R$ of any element of the tail maps to $t_{e,0,0}$ in $B$ and $t'_{e,0,0}$ in $B'$.

    \item For any $j<y_{e}$, and any $i$, define $S(t_{e,i,j})=t_{e,i+1,0}$, while waiting for the tail to close.
\end{itemize}
Observe that whenever any element $t$ of the tail is enumerated into $B_{n}$, the elements $R(t)$ and $SR(t)$ should also be enumerated `quickly' into $B_{n}$. In addition, the tail also `joins' up in primitive recursive (in $B_{n}$) time to the original and duplicated $e^{th}$ component. In much the same way as before, we apply the pressing strategy to ensure that either $B_{n}\not\cong B$ (declared inactive or pending), or $B_{n}$ enumerates the tail. In addition, by choice of $y_{e}$, $y_{e}$ is strictly smaller than the eventual cycle sizes for $R(b)$ of any $e'$-witness $b$ of $B_{n}$ where $n\leq\lfloor e'/2\rfloor$, and different from any size declared unavailable thus far. Once the copies of $t_{e,0,0}$ has been enumerated in each $B_{n}$ for each $n\leq\lfloor e/2\rfloor$, we close the tail in the same way as before: for the unique $i$ on which $S(t_{e,i,0})$ is yet to be defined, let $S(t_{e,i,j})=t_{e,i,0}$ for each $j<y_{e}$.

We now explain how the various $M_{e}$ may be chosen. Recall that each structure potentially only declares one size to be unavailable, and the original components use only exactly one size each. Similarly, each tail uses also only exactly one size. Thus, by picking $M_{e+1}$ to be larger than the sum of the total number of potential tails, $e+1$, the number of components, $e+1$, and the number of structures, $\lfloor e/2\rfloor+1$, it is guaranteed that there must always be available sizes as required. More specifically, pick $M_{e+1}=4(e+1)>2(e+1)+(\lfloor e/2\rfloor+1)+1$.

\

\emph{The strategy for $Q_{m,n}$:} If $m=n$, then the requirement is trivially satisfied by the identity map. We consider, without loss of generality, the case where $m<n$. The main challenge is in recovering the stage at which both the original and duplicate components are enumerated into both $B_{m}$ and $B_{n}$. Evidently, if we do not wait till such a stage, any attempt at defining an isomorphism may fail; there is no way of knowing if the original component in $B_{m}$ maps to the original or duplicate component in $B_{n}$.

Recall that the amount of time $B$ spends enumerating the $e^{th}$ component depends on the structures $B_{k}$ for each $k\leq\lfloor e/2\rfloor$. Since each of these structures are punctual, and $g$ primitive recursively computes every primitive recursive function ($g(\langle e+1,x\rangle)$ is the stage at which $p_{e}(x)$ converges), then we should be able to recover the stages at which each component in $B$ closes. In much the same vein, we may also recover the stage at which the tail is fully enumerated into $B$ and similarly, each $B_{k}$. (This will be verified formally later.) Thus, we may obtain stages large enough at which $B_{m}$ and $B_{n}$ has completed enumerating the full `local' structure, allowing us to define the isomorphism in the obvious way.

\

\emph{Construction:} At each stage of the construction, there will be exactly one component, or tail of a component which is not closed in $B$ and $B'$; there is exactly one pair $e,i$, such that $S(b_{e,i,0})$ and $S(b'_{e,i,0})$ has yet to be defined. Recall that $M_{e}$ is defined as $4e$ for each $e$ and that an element $b$ of $B_{n}$ is an $e$-witness if for each $0<m\leq M_{e+1}$, $C^{m}R(b)\downarrow\neq R(b)$. We also say that $B_{n}$ is pending if there is some $e$-witness for $B_{n}$.

During stage $s$, let $e$ be the least such that one of the following holds.
\begin{enumerate}[label=(\roman*)]
    \item All components and their tails currently in $B$ are closed, and $b_{e,0,0}$ has yet to be enumerated into $B$.

    \item The $e^{th}$ component is not closed in $B$.

    \item All components and their tails currently in $B$ are closed, and the $e^{th}$ component is ready.

    \item The tail of the $e^{th}$ component is not closed in $B$.
\end{enumerate}

\begin{description}
    \item[Case (i)] If (i) holds, then recursively assume that for each $n\leq\lfloor(e-1)/2\rfloor$, if $B_{n}$ has been declared pending, then it has a $e-1$-witness. Let $0<x_{e}<M_{e+1}$ be the least size which is currently available. Use $x_{e}$ as the cycle size for the $e^{th}$ component and declare it unavailable (it will not be used for any other component or tail). Begin the $e^{th}$ component in both $B$ and $B'$, enumerating $b_{e,0,0}$ and $b'_{e,0,0}$ with cycle sizes $x_{e}$ respectively. Leave $S(b_{e,0,j})$ and $S(b'_{e,0,j})$ undefined for each $j<x_{e}$.

    \item[Case (ii)] If (ii) holds, then let $i$ be the largest index such that $S(b_{e,i,0})$ has yet to be defined. For each $n\leq\lfloor e/2\rfloor$, where $B_{n}$ is not inactive, check if the following holds.
    \begin{itemize}
        \item If $B_{n}$ is pending, then $b$ has a $e$-witness.
    
        \item If $B_{n}$ is not pending, then it has enumerated some element $b$ such that the least $m>0$ for which $C^{m}R(b)\downarrow=R(b)$ is exactly $x_{e}$.
    \end{itemize}
    If the above is true for all $B_{n}$ which are not inactive, then close the $e^{th}$ component by defining $S(b_{e,i,j})=b_{e,i,0}$ in $B$ and $S(b'_{e,i,j})=b'_{e,i,0}$ in $B'$ for each $j<x_{e}$. If the above fails for some $n\leq\lfloor e/2\rfloor$, then continue extending the $e^{th}$ component by enumerating $b_{e,i+1,0}$ with cycle size $x_{e}$ and defining $S(b_{e,i,j})=b_{e,i+1,0}$. Also continue computing $C^{m}R(b)$ for each $b\in B_{n}$ and each $m\leq M_{e+1}$. As discussed previously one of the following must happen.
    \begin{itemize}
        \item $B_{n}$ enumerates the root of the $e^{th}$ component: $B_{n}$ enumerates an element $R(b)$ and the least $m>0$ for which $C^{m}R(b)\downarrow=R(b)$ is exactly $x_{e}$.
        
        \item Some $e$-witness for $B_{n}$ is found, then we may declare $B_{n}$ pending (if it was not previously pending) and proceed to the next required action.
    
        \item $C^{m}R(b)\downarrow=R(b)$ for some $0<m\leq M_{e+1}$ but $m\neq x_{e}$. There can only possibly be finitely many such $b$ found in $B_{n}$ such that $m=x_{e'}$ or $y_{e'}$ for some $e'<e$. After which $B_{n}$ will be declared inactive. The only other possibility is that $B_{n}$ enumerates some $b$ for which the least $0<m\leq M_{e+1}$ where $C^{m}R(b)\downarrow=R(b)$ is different from all $x_{e'}$ and $y_{e'}$ for $e'<e$. If such a $b$ is found, then declare $m$ to be unavailable, and $B_{n}$ to be inactive.
    \end{itemize}
    That is, we eventually close the $e^{th}$ component, and when we do, all currently pending $B_{n}$ has a $e$-witness.
    
    \item[Case (iii)] If (iii) holds, then let $e'$ be the largest such that $b_{e',0,0}$ has been enumerated into $B$. Pick $y_{e}<M_{e'+1}$ to be the least size that is available and subsequently declare it unavailable (this ensures that $y_{e}$ is different from all earlier cycles and all future cycles in different components or tails). Perform the switch for the $e^{th}$ component by enumerating $t_{e,0,0},t_{e,1,0}$, the duplicate $e^{th}$ component, and their counterparts into $B$ and $B'$ respectively. Let the cycle size of $t_{e,0,0}$ and $t_{e,1,0}$ be $y_{e}$. For each $j<y_{e}$, define $S(t_{e,0,j})=t_{e,1,0}$ and leave $S(t_{e,1,j})$. As discussed in the strategy for $G$, define for each $i<2$, $P(t_{e,i,0})=d_{e,i,0}$ in $B$, and $P(t'_{e,i,0})=d'_{e,1-i,0}$ in $B'$.

    \item[Case (iv)] If (iv) holds, then let $i$ be the largest such that $S(t_{e,i,0})$ is currently undefined. Just as in Case (ii), we only close the tail when for each $n\leq\lfloor e/2\rfloor$ where $B_{n}$ is not inactive, the following holds.
    \begin{itemize}
        \item If $B_{n}$ is pending, then $b$ has a $e'$-witness where $e'$ is currently the largest for which $b_{e',0,0}$ exists in $B$.

        \item $B_{n}$ has enumerated some element $b$ such that the least $m>0$ for which $C^{m}R(b)\downarrow=R(b)$ is exactly $y_{e}$.
    \end{itemize}
    For the same reasons as before, we eventually close the $e^{th}$ tail, and when we do, all currently pending $B_{n}$ has a $e'$-witness.
\end{description}

\

\emph{Verification:} We first show that $B$ and $B'$ are punctual structures. The functions $P,R$ and $C$ are all defined immediately whenever some element is enumerated into $B$ or $B'$. The only function which is delayed is $S$ when we do not want to close the $e^{th}$ component or tail yet. However, the delay lasts only for a single stage as seen in Cases (ii) and (iv) of the construction: for the elements $b_{e,i,j}$ or $t_{e,i,j}$ for which $S$ is undefined at the end of stage $s$, the $e^{th}$ component or tail either closes at stage $s+1$, or a new element $b_{e,i+1,0}$ or $t_{e,i+1,0}$ is enumerated to serve as the $S$-image. Thus, $B$ and $B'$ are punctual.

Now we show that if $B_{n}$ is ever declared inactive or pending, then $B_{n}\not\cong B$. If $B_{n}$ is declared inactive at some stage, then one of the following must hold.
\begin{itemize}
    \item There are distinct $b_{0},b_{1},\dots,b_{N}$ and some $x_{e}$ such that for each $i\leq N$, the least $m>0$ for which $C^{m}R(b_{i})=R(b_{i})$ is exactly $x_{e}$ and $N$ is strictly larger than the total number of elements in the closed $e^{th}$ component in $B$. Since each component has a distinct cycle size, and also different from all potential and existing tails, $B_{n}$ cannot possibly be isomorphic to $B$.

    \item There is some element $b\in B_{n}$ such that the least $m>0$ for which $C^{m}R(b)\downarrow=R(b)$ is different from the size of any cycle currently in $B$. From Case (ii) of the construction, if such a $b$ is found, we must have defined $m$ to be unavailable, and thus never enumerate any cycle of the same size into $B$.
\end{itemize}
In other words, if $B_{n}$ is ever declared inactive, then it must be that $B_{n}\not\cong B$. We may thus suppose that $B_{n}$ is never declared inactive and hence stays pending forever. From the construction, a structure $B_{n}$ is first declared pending if for some $e$, an $e$-witness for $B_{n}$ is found. The claim now is that this $e$-witness $b$ for $B_{n}$ eventually becomes an $e'$-witness for all $e'$.

Inductively assume that $b$ is an $e'$-witness. That is, for each $0<m\leq M_{e'+1}$, $C^{m}R(b)\downarrow\neq R(b)$. From Case (ii) of the construction, the $e'+1$-th component is only closed after any pending structure has obtained an $e'+1$-witness. It remains to show that $b$ must become an $e'+1$-witness. Since $x_{e'+1}<M_{e'+1}$ and $b$ is an $e'$-witness, then the cycle size of $R(b)$ (in $B_{n}$) cannot possibly mimic that of any cycle in the $e'+1$-th component of $B$. If there is some $0<m\leq M_{e'+2}$ for which $C^{m}R(b)\downarrow= R(b)$, then the least $m>0$ where this hold is necessarily larger than $M_{e'+1}$ and thus bigger than any cycle size shown in $B$ thus far. In particular, $m$ would then be declared as unavailable and $B_{n}$ would become inactive. That is, it must be that $C^{m}R(b)\downarrow\neq R(b)$ for each $0<m\leq M_{e'+2}$, and thus, $b$ must also be an $e'+1$-witness. Therefore, $B_{n}$ cannot possibly be isomorphic to $B$ as $B_{n}$ contains an infinite component.

\

\emph{$Q_{m,n}$ is satisfied:} If $m=n$, then the identity map is obviously a punctual isomorphism. Without loss of generality, we may assume that $m<n$. By the premise of the requirement, $B_{m}\cong B_{n}\cong B$, and thus, it suffices to produce an isomorphism on each component. In particular, we need only compute a stage at which $B_{m}$ and $B_{n}$ has fully enumerated both the original and duplicate components together with the tail.

We claim that given $g$ and $e$, it is primitive recursive to find the stage at which the $e^{th}$ component closes in $B$ and $B'$, and thus the total number of elements within the $e^{th}$ component. Recall that the $e^{th}$ component in $B$ and $B'$ closes exactly when there is some $b\in B_{k}$ for each $k\leq\lfloor e/2\rfloor$ such that one of the following holds (see Case (ii) of the construction).
\begin{enumerate}[label=(\roman*)]
    \item The least $0<l\leq M_{e+1}$ for which $C^{l}R(b)\downarrow=R(b)$ is exactly $x_{e}$ (the cycle size of the $e^{th}$ component).

    \item The least $0<l\leq M_{e+1}$ for which $C^{l}R(b)\downarrow\neq R(b)$ is different from $x_{e'}$ and $y_{e'}$ if they exist, for each $e'\leq e$. ($B_{n}$ will then be declared inactive.)

    \item For all $0<l\leq M_{e+1}$, $C^{l}R(b)\downarrow\neq R(b)$. ($b$ is a $e$-witness for $B_{n}$.)
\end{enumerate}
Since $g$ primitive recursively computes all primitive recursive functions, and $B_{k}$ is punctual, we may assume that we have access to the versions of $S,P,R$ and $C$ in each $B_{k}$. In particular, using $g$, we apply the following procedure.
\begin{itemize}
    \item First compute the stage at which $B$ and $B'$ closes the $0^{th}$ component, or equivalently, the stage at which some $b\in B_{0}$ is found satisfying one of (i), (ii) or (iii) for $e=0$. This is clearly primitive recursive in $g$, as we may use $g$ to compute the actions of the versions of the maps $R,C$ of $B_{0}$, and we need only check if $CR(b)\downarrow=R(b)$ for the first $b$ enumerated into $B_{0}$. In fact, $g$ provides the stage at which $CR(b)\downarrow$, and at such a stage, we must have closed the $0^{th}$ component. In particular, the number of elements in the $0^{th}$ component is primitive recursive in $g$ and $0$.

    \item Suppose recursively that the number of elements in the original $e'$-th components of $B$ and $B'$ for all $e'<e$ may be computed primitive recursively in $g$ and $e'$. Since each $B_{k}$ is punctual, we may compute using $g$, the settling time of each of the functions of $B_{k}$ on the first $N$ many elements of $B_{k}$, where $N$ is the total number of elements in the first $e$ many original components of $B$ and $B'$. After the delay given by this settling time, it is evident that $B_{k}$ must have enumerated some element $b$ that is not part of any of the $e'$-th component for any $e'<e$. By further computing the settling time (using $g$) for $C^{m+1}R(b)$ on such a $b$, we may then obtain the stage at which some $b\in B_{k}$ is discovered to satisfy one of (i), (ii), or (iii). Repeating this for each $k\leq\lfloor e/2\rfloor$ allows us to then obtain the stage at which the $e^{th}$ component is closed in $B$ and $B'$. Thus, such a stage can be computed primitive recursively in $g$ and $e$.
\end{itemize}
By applying a similar argument, we may further obtain primitive recursively, using $g$ and $e$, a stage at which the copy of the $e^{th}$ component in $B_{k}$ is closed or that $B_{k}$ has been declared inactive or pending for each $k\leq\lfloor e/2\rfloor$. Similarly, it is primitive recursive in $g$ and the stage at which $t_{e,0,0}$, say $s_{e}$, is enumerated into $B$ to compute the stage at which the $e^{th}$ tail is closed in $B$.

We now show that $s_{e}$ may be obtained primitive recursively using $g$ and $e$. If $e$ is even, then recall that the $e^{th}$ component becomes ready at the stage at which the $e+2$-th component closes. Applying the fact above allows us to conclude that this can be computed primitive recursively using $g$ and $e$. If $e$ is odd, then the $e^{th}$ component becomes ready at the stage given by $g((e-1)/2)$. Once the stage at which a component becomes ready is obtained, we need only further compute the stages at which any component of smaller index completes switching, or equivalently, the stages at which the tails of smaller index closes. This is clearly primitive recursive in $g$ and the stage at which the $e^{th}$ component becomes ready. In conclusion, given $g$ and $e$, we may primitive recursively compute a stage large enough such that $B_{k}$ for any $k\leq\lfloor e/2\rfloor$ has enumerated every element of the original and duplicate $e^{th}$ components, and also the $e^{th}$ tail. Applying this for every $e\geq 2n+2$ allows us to conclude that $g$ may compute an isomorphism between $B_{m}$ and $B_{n}$ for each component with index $e\geq 2n+2$. Since each component and tail is finite, we may non-uniformly fix the isomorphism between $B_{m}$ and $B_{n}$ for the components with index $<2n+2$. Thus, $g\geq_{pr}h\oplus h^{-1}$ for some $h:B_{m}\to B_{n}$ an isomorphism.

\

\emph{Computing $g(x)$:} Let $f:B\to B'$ be an isomorphism. Given $x$, compute $g(x)$ as follows. For each $e\leq x$ Compute $f(b_{2e,0,0})$. Since the tail is rigid, $f$ must map $t_{2e,0,0}$ in $B$ to $t'_{2e,0,0}$ in $B'$. In particular, $f(b_{2e,0,0})=fP(t_{2e,0,0})=Pf(t_{2e,0,0})=P(t'_{2e,0,0})=d'_{2e,0,0}$. From Case (iii) of the construction, $d'_{2e,0,0}$ is enumerated into $B'$ only at a stage after the $2e^{th}$ component is ready; this is a stage at least as large as the stage where the $2e+2$-th component is closed in $B$. Thus, $f(b_{2e,0,0})$ provides a stage larger than the index of $b_{2e+2,0,0}$ and necessarily also $b_{2e+1,0,0}$. Taking $e=x$, we may thus obtain the index of $b_{2x+1,0,0}$ primitive recursively in $f$. For similar reasons as before, $f(b_{2x+1,0,0})=fP(t_{2x+1,0,0})=Pf(t_{2x+1,0,0})=P(t'_{2x+1,0,0})=d'_{2x+1,0,0}$ which must have index larger than the value of $g(x)$. By enumerating $W$ or computing the appropriate primitive recursive function up to the index of $d'_{2x+1,0,0}$ allows us to recover the exact value of $g(x)$. Theorem~\ref{thm:deg} is proved.
\end{proof}

\bibliographystyle{alpha}
\bibliography{mybib}

\end{document}